\documentclass{amsart}

\usepackage{amsmath, amssymb, amsthm}
\usepackage{color}
\usepackage{url}

\newtheorem{theorem}{Theorem}[section]
\newtheorem{lemma}[theorem]{Lemma}

\newtheorem{claim}[theorem]{Claim}

\newtheorem{question}[theorem]{Question}
\newtheorem*{THM}{Main Theorem}
\newtheorem*{THM'}{Main Theorem (version 2)}
\newtheorem*{THM''}{Theorem}
\newtheorem*{Claim 53}{Claim 5.3'}

\theoremstyle{definition}
\newtheorem{definition}[theorem]{Definition}
\newtheorem{remark}[theorem]{Remark}
\newtheorem{example}[theorem]{Example}

\newcommand{\dom}{\mathrm{dom}}
\newcommand{\bb}{\mathbb}

\newcommand{\otp}{\mathrm{otp}}

\newcommand{\tp}{\mathrm{tp}}

\newcommand{\ord}{\mathrm{Ord}}
\newcommand*\oline[1]{%
  \vbox{%
    \hrule height 0.5pt
    \kern0.25ex
    \hbox{%
      \kern -0.2em
      \ifmmode#1\else\ensuremath{#1}\fi
      \kern 0em
    }
  }
}

\title{Simultaneously vanishing higher derived limits}

\author{Jeffrey Bergfalk}
\address{Universit\"{a}t Wien \\
Institut f\"{u}r Mathematik \\
Kurt G\"{o}del Research Center \\
Kolingasse 14-16 \\
1010 Wien, Austria}
\email{jeffrey.bergfalk@univie.ac.at}

\author{Chris Lambie-Hanson}
\address{Department of Mathematics and Applied Mathematics \\
Virginia Commonwealth University \\
Richmond, VA 23284 \\ United States}
\email{cblambiehanso@vcu.edu}

\begin{document}
\begin{abstract}
  In 1988, Sibe Marde\v{s}i\'{c} and Andrei Prasolov isolated an inverse system $\mathbf{A}$
  with the property that the additivity of strong
  homology on any class of spaces which includes the closed subsets of
  Euclidean space would entail that $\lim^n\mathbf{A}$ (the $n^{\text{th}}$ derived limit of $\mathbf{A}$) vanishes for every $n >0$. Since that time, the question of whether it is consistent with the $\mathsf{ZFC}$ axioms that $\lim^n \mathbf{A}=0$ for every $n >0$ has remained open. It remains possible as well that this condition in fact implies that strong homology is additive on the category of metric spaces.

   We show that, assuming the existence of a weakly compact cardinal, it is indeed consistent with the $\mathsf{ZFC}$ axioms that $\lim^n \mathbf{A}=0$ for all $n >0$. We show this via a finite support iteration of Hechler forcings which is of weakly compact length. More precisely, we show that in any forcing extension by this iteration a condition equivalent to $\lim^n\mathbf{A}=0$ will hold for each $n>0$. This condition is of interest in its own right; namely, it is the triviality of every coherent $n$-dimensional family of certain specified sorts of partial functions $\mathbb{N}^2\to\mathbb{Z}$ which are indexed in turn by $n$-tuples of functions $f:\mathbb{N}\to\mathbb{N}$. The triviality and coherence in question here generalize the classical and well-studied case of $n=1$.

\end{abstract}



\maketitle

\section{Introduction}\label{introduction}

A main and organizing theme in the study of infinitary combinatorics is the phenomenon of \emph{incompactness}; broadly speaking, the term denotes settings in which the local and global behaviors of a structure sharply diverge. \emph{Nontrivial coherent families of functions} on a variety of index-sets form one prominent class of  examples. These connect in turn to local-to-global questions in homology theory, wherein the nontrivial coherence relations of such set-theoretic interest figure as only the first in an infinite family of incompactness principles given by the derived functors $\lim^n$ of the inverse limit of various inverse systems. For one particular inverse system, denoted $\mathbf{A}$ below, the question of whether these associated incompactness principles can simultaneously fail had been longstanding; the main result of this paper is that under the assumption of the existence of a weakly compact cardinal, they can. This result carries implications for the strong homology of metric spaces and, as we note below, possibly for other areas of mathematics as well.

We begin by reviewing the historical background to this result; it traces to Sibe Marde\v{s}i\'{c} and Andrei Prasolov's 1988 paper ``Strong homology is not additive" \cite{mp}. Their title references the following potential continuity property of a homology theory:
\begin{definition}[\cite{milnorax}] A homology theory is \emph{additive on the class $\mathcal{C}$ of topological spaces} if for every natural number $p$ and every family $\{X_\alpha\,|\,\alpha\in A\}$ with each $X_\alpha$ and $\coprod_A X_\alpha$ in $\mathcal{C}$, the map
 $$i_p:\bigoplus_A \mathrm{H}_p(X_\alpha)\rightarrow \mathrm{H}_p(\coprod_A X_\alpha)$$
 induced by the inclusion maps $\iota_\alpha:X_\alpha\hookrightarrow\coprod_A X_\alpha$ is an isomorphism.
 \end{definition}

In \cite{mp}, Marde\v{s}i\'{c} and Prasolov isolated an inverse system $\mathbf{A}$ with the property that the additivity of strong homology on any class of topological spaces which includes the closed subsets of Euclidean space would entail that $\lim^n\mathbf{A}=0$ for all $n>0$. They succeeded also in casting the vanishing of $\lim^1\mathbf{A}$ in the following combinatorial terms: $\lim^1\mathbf{A}=0$ if and only if for every family of functions
$$\Phi=\left\langle\varphi_f:\{(k,m)\mid m\leq f(k)\}\to\mathbb{Z}\mid f\in\,\!^\omega\omega\right\rangle$$ whose elements agree pairwise modulo finite sets there exists some $\psi:\omega\times\omega\to\mathbb{Z}$ agreeing mod finite with each function in $\Phi$. More succinctly, in a parlance that has since grown standard, $\lim^1\mathbf{A}=0$ if and only if every \emph{coherent} $\Phi$ as above is \emph{trivial}. By way of this characterization, Marde\v{s}i\'{c} and Prasolov showed that the continuum hypothesis implies that $\lim^1\mathbf{A}\neq 0$. It follows that it is consistent with the $\mathsf{ZFC}$ axioms that strong homology is not additive, not even on the class of closed subspaces of $\mathbb{R}^2$.

The following year, Alan Dow, Petr Simon, and Jerry Vaughan showed that the Proper Forcing Axiom implies that $\lim^1\mathbf{A}=0$, underscoring the possibility that the additivity of strong homology may \textit{also} be consistent with the $\mathsf{ZFC}$ axioms \cite{dsv}. Soon thereafter, Stevo Todorcevic showed that the Open Coloring Axiom implies that $\lim^1\mathbf{A}=0$, while Martin's Axiom does not \cite{todpp, todcmpct}. Around 2000, Andrei Prasolov produced a nonmetrizable $\mathsf{ZFC}$ counterexample to the additivity of strong homology \cite{pra}, but the additivity question on any ``nicer'' class --- Polish or locally compact metric spaces, for example, or even metric spaces outright --- remained entirely open (and for the aforementioned classes remains so; see, however, Remark \ref{recent_work_remark} for recent progress on this question).\footnote{Prasolov's example was, in essence, a geometric ``realization'' of an $(\omega_1,\omega_1^{*})$-Hausdorff gap.} In short, in this sequence, the vanishing of $\lim^n\mathbf{A}$ became a topic of set-theoretic study in its own right \cite{kamo, far1}, one closely linked to the study of forcing axioms and one, furthermore, rekindling older lines of research into relations between homological dimension and the cardinals $\{\aleph_n \mid n \in \omega\}$ \cite{osof,blh}.
It was therefore natural that the consistency of the additivity of strong homology would close out the list of  open questions in Justin Moore's 2010 ICM survey on the Proper Forcing Axiom \cite{moorepfa}. Moore further observed therein that ``it is entirely possible that it is a theorem of $\mathsf{ZFC}$ that either [$\lim^1\mathbf{A}\neq 0$] or [$\lim^2\mathbf{A}\neq 0$].'' The first set-theoretic computation of $\lim^2\mathbf{A}$ appeared some six years later in \cite{b}; here framings of $\lim^n\mathbf{A}$ $(n>1)$ in terms of higher-dimensional coherence were given and applied to show that the Proper Forcing Axiom implies that $\lim^2\mathbf{A}\neq 0$.
Here also Goblot's work \cite{goblot} was applied to show that $\mathfrak{d} \leq
\aleph_m$ implies that $\lim^n \mathbf{A} = 0$ for all $n > m$. Still, Moore's speculation remained unanswered.

Against this background, we may state our main result:

\begin{THM} Let $\kappa\in V$ be a weakly compact cardinal and let $\mathbb{P}$ denote a length-$\kappa$ finite-support iteration of Hechler forcings. Then $$V^{\mathbb{P}}\vDash\textnormal{``}\,\mathrm{lim}^{n}\mathbf{A}=0\text{ for all }n>0.\textnormal{''}$$
\end{THM}

In addition to answering well-studied set-theoretic questions, the theorem is of interest for specifying a $\mathsf{ZFC}$ model --- namely, $V^{\mathbb{P}}$ --- in which strong homology may well turn out to be additive on some ``nice'' class of spaces properly containing the class of homotopy types of CW-complexes. Moreover, although we have followed tradition in foregrounding the additivity question above, the nonvanishing of $\lim^n\mathbf{A}$ for some $n>0$ is also the only known obstacle to strong homology \emph{having compact supports} on some such ``nice'' class; hence this property of strong homology conceivably holds in $V^{\mathbb{P}}$ as well.

\begin{remark} \label{recent_work_remark}
  We have largely retained the text of this article's original incarnation, but a
  couple of recent developments should be noted.

  First, work of the first author together with Nathaniel Bannister and Justin Tatch
  Moore \cite{bannister_bergfalk_moore} vindicated the speculations of the previous paragraph by
  showing that, in $V^{\mathbb{P}}$, strong homology is indeed additive and has compact supports on
  the class of locally compact separable metric spaces. For the property of having compact supports,
  this class of spaces is at least close to optimal, by the results of \cite{gunther} and \cite{lisica},
  but the question of whether strong homology is consistently additive on some broader class of
  spaces remains open, as we discuss at greater length in our conclusion below.

  Second, work of the authors together with Michael Hru\v{s}\'{a}k \cite{svhdlwolc}, building on the
  techniques of this paper, shows that the assumption of a weakly compact
  cardinal is not necessary to obtain the conclusion of our main theorem. More
  precisely, $\lim^n\mathbf{A} = 0$ for all $n > 0$ after adding $\beth_\omega$-many
  Cohen reals to any model of $\mathsf{ZFC}$.

  For interested readers, the most immediate point is probably the following: \cite{bannister_bergfalk_moore} and \cite{svhdlwolc} each couple multiple new ideas to techniques which first appear, and, consequently, receive their fullest and most direct treatment, in the present work. Given the complexity of these techniques, there will be benefits to reading this work before either of the others, though ultimately, of course, readers may and should order their approach to these three texts most fundamentally according to their interests.
\end{remark}

These limits also bear on questions seemingly remote from strong homology: within the context of Dustin Clausen and Peter Scholze's recently-developed \emph{condensed mathematics} \cite{condensed}, $\mathbf{A}$ is the most basic system in a family of inverse systems $\mathbf{A}_{\kappa,\lambda}$ upon whose limits' simultaneous vanishing the full and faithful embedding of the derived category of pro-abelian groups into the associated derived condensed category depends \cite{email}. More generally speaking, the above theorem reflects new levels of insight into the set-theoretic content of derived functors, and as such takes its place in a line of research beginning with Shelah's solution to Whitehead problem \cite{sheIAG, AFM}. It bears comparison, lastly, with Boban Velickovic and Alessandro Vignati's recent result in the opposite direction:
\begin{THM''}[\cite{VV}]
For all integers $n\geq 0$ it is consistent with the $\mathsf{ZFC}$ axioms that $\lim^n\mathbf{A}\neq 0$.
\end{THM''}

As in so much of the above-described research history, our argument of a fundamentally algebraic result will be predominantly set-theoretic in nature. Our Main
Theorem, in particular, can and will be recast as a purely set-theoretic statement that we
feel is of interest in its own right. For this reason we have divided
Section \ref{main}, in which we fix our notational conventions and record the definitions
and background facts most relevant to our main result, into two subsections.
One is more set-theoretic and one is more cohomological in character. Section \ref{set_background} contains the set-theoretic
reformulation of our Main Theorem and must be read in order to understand the
arguments of subsequent sections. Section \ref{hom_background} records some of the original context of the problems
addressed in this paper together with the argument that our set-theoretic reformulation
of the Main Theorem is in fact equivalent to its statement above. Readers less familiar with
homological algebra may safely skip most of Section \ref{hom_background} without sacrificing
any understanding of the remainder of the paper (we indicate at the end of
Section \ref{set_background} exactly which part of Section \ref{hom_background}
is necessary for the rest of the paper). Of course, such readers are invited to return to this subsection after reading that remainder, in order to better contextualize its contents.

The paper thereafter is structured as follows.
In Section \ref{multi} we describe a multidimensional
$\Delta$-system lemma lying at the core of our subsequent arguments. We then
turn to the proof of our Main Theorem. Due to the technical
complexity of the full proof of the theorem, we begin by
presenting two cases in which our argument's core ideas more transparently appear. In Section \ref{lim_1_section}, after recording the
requisite facts about finite-support iterations of Hechler forcing, we prove the
$n=1$ case of our main theorem. We prove the $n=2$ case in Section \ref{lim_2 section}.
The full proof of the Main Theorem is contained in Section \ref{lim_n section}.
We conclude with a brief discussion of the import of this result and with
the questions following most immediately in its wake.


We close this introduction with a more general word on notations and conventions. For any $X$ and cardinal $\lambda$ we write $[X]^\lambda$ to denote the collection of subsets of $X$ of cardinality $\lambda$ and $[X]^{<\lambda}$ for $\bigcup_{\kappa<\lambda}[X]^\kappa$. In particular we view $[X]^0$ as $\{\emptyset\}$. When $X$ is a collection of ordinals, it is frequently convenient to regard elements $u$ of $[X]^{<\omega}$ as finite increasing sequences, and vice versa. For such $u$ and $i <|u|$ we let $u(i)$ denote the unique $\alpha \in u$ such that $|u \cap \alpha| = i$. The notation $\vec{u}$ will always stand for an ordered tuple of the form $\langle u_0,\dots,u_k \rangle$, though we will occasionally begin our indexing with $1$. Also, if $u$, $v$, and $w$ are finite sets of ordinals, then the statement $w = u ^\frown v$ indicates both that $w = u \cup v$ and that $u < v$, i.e.,
that every ordinal in $u$ is less than every ordinal in $v$, and the
statement $u \sqsubseteq v$ indicates that $u$ is an initial segment of
$v$. We use angled brackets to emphasize the indexed or ordered character of a set; in all of this, though, the surest guide will simply be context.

We follow \cite{baumITF} and \cite{kunen} in our approach to forcing. We remark
though that we view conditions in a $\kappa$-length finite support iteration
as finite partial functions $p$ with domains contained in $\kappa$ rather than as
total functions $p$ such that $p(\alpha)$ is a name for the trivial condition
for all but finitely many $\alpha < \kappa$. The difference between these views, of course, is cosmetic.

\section{Main definitions and conventions}\label{main}

\subsection{Set theoretic background} \label{set_background}

Our primary objects of study are families of functions from subsets of
$\omega^2$ to $\bb{Z}$. Let us begin by introducing some basic definitions and
notational conventions.

Given functions $f,g : \omega \rightarrow \omega$, let $f \leq g$ if and only if
$f(i) \leq g(i)$ for all $i \in \omega$, let $f \leq^* g$ if and only if
$f(i) \leq g(i)$ for all but finitely many $i \in \omega$, and let $f =^* g$
if and only if $f(i) = g(i)$ for all but finitely many $i \in \omega$. We let
$f \wedge g$ denote the greatest lower $\leq$-bound of $f$ and $g$, i.e.,
$(f \wedge g)(i) = \min(\{f(i), g(i)\})$ for all $i \in \omega$. Similarly,
if $\vec{f} = \langle f_0, \ldots, f_n \rangle$ is a sequence of elements of ${^\omega}\omega$,
then $\bigwedge \vec{f}$ denotes the greatest lower $\leq$-bound of the functions
$f_0, \ldots, f_n$. If $f \in {^\omega}\omega$, then $I(f)$ denotes the set
$\left\{(i,j) \in \omega^2 ~ \middle| ~ j \leq f(i)\right\}$; visually, this is the region
below the graph of $f$. Relatedly, $U(f)$ denotes the set
$\left\{g \in {^\omega}\omega ~ \middle| ~ g \leq f\right\}$. We note that, for
ease of readability, we will sometimes write $\wedge \vec{f}$ in place of
the more formally correct $\bigwedge \vec{f}$ in expressions such as
$\varphi: I(\wedge \vec{f}) \rightarrow \bb{Z}$. The correct interpretation
will always be clear from context.

Our interest is in families of functions indexed by elements of
$({^\omega}\omega)^n$ for some positive integer $n$. Before giving general definitions, we recall the better-known special case in which
$n = 1$.

\begin{definition} \label{cohtriv}
  A family of functions $\Phi = \left\langle\varphi_f : I(f) \rightarrow \bb{Z} ~ \middle|
  ~ f \in {^\omega}\omega \right\rangle$ is \emph{coherent} if
  \[
    \varphi_g \restriction I(f \wedge g) - \varphi_f \restriction I(f \wedge g)
    =^* 0
  \]
  for all $f,g \in {^\omega}\omega$.

  The family $\Phi$ is \emph{trivial} if there exists a function $\psi : \omega^2
  \rightarrow \bb{Z}$ such that
  \[
    \psi \restriction I(f) - \varphi_f =^* 0
  \]
  for all $f \in {^\omega}\omega$. In this case, we say that $\psi$ \emph{trivializes} $\Phi$.

  The notions of \emph{coherent} and \emph{trivial} also apply to families
  of functions indexed by some \emph{subset} of ${^\omega} \omega$ in the
  obvious way.
\end{definition}

\begin{remark} \label{restriction_remark}
  In the following, as in Definition \ref{cohtriv}, sums involving functions with
  different domains will be common. In the interests of readability, we will tend to refrain from notating the restrictions of such functions to
  the intersection of their domains. An equation like
  \[
    \varphi_g \restriction I(f \wedge g) - \varphi_f \restriction I(f \wedge g)
    =^* 0,
  \]
  for example, will more typically appear as
  \[
    \varphi_g - \varphi_f =^* 0
  \]
  hereafter.
\end{remark}

Clearly any trivial family of functions is coherent. More interesting are those families of functions which are coherent but not trivial; in these families there is a tension between local and global behaviors which is exemplary of the broader set-theoretic theme of \emph{incompactness}. More precisely, observe that any coherent $\Phi = \left\langle\varphi_f ~ \middle| ~ f \in {^\omega}\omega\right\rangle$ is ``locally'' trivial: for any $g \in {^\omega}\omega$, the function $\varphi_g$ trivializes the family
$\left\langle\varphi_f ~ \middle| ~ f \in U(g)\right\rangle$.  A nontrivial coherent $\Phi$ is simply one in which these local phenomena do not globalize. Such families are the subjects of the works \cite{dsv}, \cite{todcmpct},
\cite{kamo}, and \cite{far1}, among others.
As we will see in Section \ref{hom_background}, their existence is equivalent to the statement $\lim^1 \mathbf{A} \neq 0$.

For the more general definitions of $n$-coherence and $n$-triviality, we need some more notation.

\begin{definition}
  Suppose that $n$ is a positive integer and $\vec{f} = \langle f_0, \ldots f_n\rangle$ is a
  sequence of length $n + 1$. If $i \leq n$, then $\vec{f}^i$ denotes the
  sequence of length $n$ obtained by removing the $i^{\mathrm{th}}$ entry of
  $\vec{f}$, i.e., $\vec{f}^i = \langle f_0, \ldots, f_{i-1}, f_{i + 1}, \ldots, f_n\rangle$.

  If $\sigma$ is a permutation of $\langle 0, \ldots, n \rangle$, then $sgn(\sigma)$ denotes the
  \emph{sign} or \emph{parity} of $\sigma$ (so $sgn(\sigma)$ is either $1$ or $-1$).
  We will use the notation $\sigma(\vec{f})$ to denote the sequence
  $\langle f_{\sigma(0)}, \ldots, f_{\sigma(n)}\rangle$.
\end{definition}

\begin{definition} \label{NCOH}
  Let $n$ be a positive integer, and let
  \[
    \Phi = \left\langle\varphi_{\vec{f}} : I(\wedge \vec{f}\hspace{2pt}) \rightarrow \bb{Z} ~
    \middle| ~ \vec{f} \in ({^\omega}\omega)^n\right\rangle
  \]
  be a family of functions.
  \begin{enumerate}
    \item $\Phi$ is \emph{alternating} if
    \[
      \varphi_{\sigma(\vec{f})} = sgn(\sigma) \varphi_{\vec{f}}
    \]
    for every $\vec{f} \in ({^\omega}\omega)^n$ and every permutation $\sigma$ of
    $(0, \ldots, n-1)$.
    \item $\Phi$ is \emph{$n$-coherent} if it is alternating and if
    \[
      \sum_{i = 0}^n (-1)^i \varphi_{\vec{f}^i} =^* 0
    \]
    for all $\vec{f} \in ({^\omega}\omega)^{n+1}$. (Note that, in accordance
    with Remark \ref{restriction_remark}, for readability we have omitted the restrictions
    of the functions in the above expression to the intersection of their domains.
    Formally, each $\varphi_{\vec{f}^i}$ should be $\varphi_{\vec{f}^i} \restriction
    I(\wedge \vec{f})$.)
    \item For $n > 1$, $\Phi$ is \emph{$n$-trivial} if there exists an alternating family
    \[
      \Psi = \left\langle\psi_{\vec{f}} : I(\wedge \vec{f}) \rightarrow \bb{Z} ~ \middle| ~
      \vec{f} \in ({^\omega}\omega)^{n-1}\right\rangle
    \]
    such that
    \[
      \sum_{i = 0}^{n-1} (-1)^i \psi_{\vec{f}^i} =^* \varphi_{\vec{f}}
    \]
    for all $\vec{f} \in (^{\omega}\omega)^n$. We term such a family an
    $n$-trivialization of $\Phi$.
  \end{enumerate}
  As in the case $n = 1$, the notions of \emph{alternating}, \emph{$n$-coherent},
  and \emph{$n$-trivial} apply in obvious ways to families
  indexed by $F^n$, where $F$ is any subset of ${^{\omega}} \omega$.
\end{definition}

Notice that if $n = 1$ then every family $\Phi$ as above is alternating, and
the definition of \textit{1-coherence} coincides with the that of \textit{coherence} in
Definition \ref{cohtriv}. Similarly, we define \emph{$1$-triviality} to coincide with \emph{triviality} as defined in Definition \ref{cohtriv}. If the value of $n$ is clear from context then
it may be dropped from the expression \emph{$n$-trivial}. Also we will sometimes
consider functions $\varphi_{\vec{f}}$ or $\psi_{\vec{f}}$ without explicitly
specifying their domains or codomains. It is in such cases implicit that
these functions have domain $I(\wedge\vec{f})$ and codomain $\bb{Z}$.

Observe that if $\Phi$ is alternating then $\varphi_{\vec{f}} = 0$ whenever $\vec{f} \in ({^\omega}\omega)^n$
has repeated entries, since if the permutation
$\sigma$ induces a simple swapping of two repeated entries in $\vec{f}$ then $sgn(\sigma) = -1$ and $\sigma(\vec{f}) = \vec{f}$ and therefore $\varphi_{\vec{f}} = - \varphi_{\sigma(\vec{f})}=-\varphi_{\vec{f}}$.

Just as when $n$ equals $1$, any $n$-trivial family of functions is
$n$-coherent. Also just as before, nontrivial $n$-coherent families are exemplary cases of set-theoretic incompactness.
This is again because if $\Phi = \left\langle\varphi_{\vec{f}} ~ \middle| ~  \vec{f} \in
({^\omega}\omega)^n \right\rangle$
is an $n$-coherent family then for each $g$ in ${^\omega}\omega$ the local family $\left\langle\varphi_{\vec{f}} ~ \middle| ~  \vec{f} \in U(g)^n \right\rangle$ is
$n$-trivial, as witnessed by the collection
\[
  \left\langle(-1)^{n+1} \varphi_{\vec{h} ^\frown \langle g \rangle} ~ \middle| ~ \vec{h} \in
  U(g)^{n-1} \right\rangle
\]
This fact will contrast with the global structure of any nontrivial such $\Phi$.

We will see in Section \ref{hom_background}
that for each positive integer $n$ the nonexistence of nontrivial
$n$-coherent families of functions is equivalent to the statement $\lim^n \mathbf{A} = 0$. Hence our Main Theorem may be rephrased as follows:

\begin{THM'}
  Let $\kappa$ be a weakly compact cardinal in $V$ and let $\bb{P}$ denote a length-$\kappa$
  finite-support iteration of Hechler forcings. Then the following holds in $V^\bb{P}$: for every
  positive integer $n$, every $n$-coherent family
  \[
    \Phi = \left\langle\varphi_{\vec{f}} : I(\wedge \vec{f}) \rightarrow \bb{Z} ~
    \middle| ~ \vec{f} \in ({^\omega}\omega)^n\right\rangle
  \]
  is $n$-trivial.
\end{THM'}

It is this version of the Main Theorem that we will prove.
To that end we give the following alternate characterization of the
$n$-triviality of an $n$-coherent family, one which is at least locally
more finitary in character than that of Definition \ref{NCOH}. Our
aim is to facilitate the analysis of these phenomena as they arise in
forcing extensions. In what follows, and throughout the paper, if $\psi:X
\rightarrow \bb{Z}$ is a function, then the \emph{support} of $\psi$ is the
set $\{x \in X \mid \psi(x) \neq 0\}$. We say that $\psi$ is \emph{finitely supported}
if the support of $\psi$ is finite.

\begin{lemma}\label{finsup}
  Suppose that $F \subseteq {^\omega}\omega$, and let $\Phi=\left\langle \varphi_{\vec{f}} ~ \middle| ~ \vec{f}\in F^n
  \right\rangle$ be an $n$-coherent
  family. Then $\Phi$ is $n$-trivial if and only if there exists an alternating
  family of finitely supported functions $\Psi=\left\langle \psi_{\vec{f}} ~ \middle| ~
  \vec{f}\in F^n\right\rangle$ such that
  \[
    \sum_{i=0}^{n}(-1)^i\varphi_{\vec{f}^i}=\sum_{i=0}^{n}(-1)^i\psi_{\vec{f}^i}
  \]
  for all $\vec{f}\in F^{n+1}$.
\end{lemma}

\begin{remark}
  See Lemma \ref{connectingmap} in Section \ref{hom_background} for a cohomological approach to the above fact; the treatment there is notably cleaner and neatly complements the more computational perspective below.
\end{remark}

\begin{proof}[Proof of Lemma \ref{finsup}]
  We first consider the case in which $n = 1$. Suppose that $\Phi$ is trivial
  and that $\tau : \omega^2 \rightarrow \bb{Z}$ is a trivialization. For
  each $f \in F$, let
  \[
    \psi_f = \varphi_f - \tau \restriction I(f).
  \]
  It is straightforward to verify that $\Psi = \left\langle \psi_f ~ \middle| ~f \in
  F \right\rangle$ is as desired.
  For the other direction, suppose that $\Psi$ is as in the statement of
  the lemma. In particular, for all $f,g \in F$ and all
  $x \in I(f \wedge g)$, we have
  \[
    \varphi_f(x) - \psi_f(x) = \varphi_g(x) - \psi_g(x).
  \]
  We can therefore define $\tau : \omega^2 \rightarrow \bb{Z}$ by setting,
  for all $x \in \omega^2$,
  \[
    \tau(x) = \varphi_f(x) - \psi_f(x)
  \]
  for some (or, equivalently, all) $f \in F$ such that $x \in I(f)$ (if there is
  no $f \in F$ such that $x \in I(f)$, simply let $\tau(x) = 0$). Since
  each $\psi_f(x)$ is finitely supported, it follows that
  $\varphi_f(x) =^* \tau \restriction I(f)$ for all $f \in F$.

  We now consider the case in which $n > 1$.
  Suppose first that $\Phi$ is $n$-trivial, and let $T = \left \langle \tau_{\vec{f}}
  ~ \middle| ~ \vec{f} \in F^{n-1}\right\rangle$ be an $n$-trivialization of
  $\Phi$. For each $\vec{f} \in F^n$, let
  \[
    \psi_{\vec{f}} = \varphi_{\vec{f}} - \sum_{i=0}^{n-1} (-1)^i \tau_{\vec{f}^i}.
  \]
  We claim that $\Psi = \left\langle\psi_{\vec{f}} ~ \middle| ~ \vec{f} \in F^n
\right\rangle$ is as desired.
  The following statements are straightforward but tedious to verify. This
  direction of the lemma being inessential to the proof of our main result, we
  leave them to the reader:
  \begin{itemize}
    \item If $T$ is an $n$-trivialization of $\Phi$ then
    each $\psi_{\vec{f}}$ is finitely supported.
    \item If both $T$ and $\Phi$ are alternating families then
    $\Psi$ also is an alternating family.
    \item For all $\vec{f} \in F^n$
    \[
      \sum_{i=0}^n (-1)^i \varphi_{\vec{f}^i} = \sum_{i=0}^n (-1)^i \psi_{\vec{f}^i}
    \]
    as desired. This can be seen by writing
    \[
      \sum_{i=0}^n (-1)^i \psi_{\vec{f}^i} = \sum_{i=0}^n (-1)^i \varphi_{\vec{f}^i}
      - \sum_{i=0}^n \sum_{j=0}^{n-1}(-1)^{i + j} \tau_{(\vec{f}^i)^j}
    \]
    and verifying that all of the terms in the double sum on the right-hand side
    of the above equation cancel out.
  \end{itemize}

  For the implication in the other direction, suppose that $\Phi$ and $\Psi$ are as in the statement of the
  lemma. For each $x \in \omega^2$, fix a function $f_x \in F$
  with $x \in I(f_x)$ if there exists such a function in $F$ (if not, then leave
  $f_x$ undefined). We will define an $n$-trivialization
  $T = \left\langle \tau_{\vec{f}} ~ \middle| ~ \vec{f} \in F^{n-1}\right\rangle$
  of $\Phi$ as follows. Given $\vec{f} \in F^{n-1}$
  and $x \in I(\wedge \vec{f})$, note that $f_x$ is defined, and let
  \[
    \tau_{\vec{f}}(x) = (-1)^n\left(\psi_{\vec{f}^\frown \langle f_x
    \rangle}(x) - \varphi_{\vec{f} ^\frown \langle f_x \rangle}(x) \right)
  \]
  The fact that $T$ is an alternating family follows immediately from the
  fact that $\Phi$ and $\Psi$ are alternating families. To see that
  $T$ is an $n$-trivialization of $\Phi$, fix $\vec{f} \in F^n$
  and let $x$ be an element of $I(\wedge \vec{f})$ outside of the support of
  $\psi_{\vec{f}}$. Then we have
  \[
    \sum_{i=0}^{n-1}(-1)^i \tau_{\vec{f}^i}(x) = (-1)^n \sum_{i=0}^{n-1}
    (-1)^i \left(\psi_{\vec{f}^i {^\frown} \langle f_x \rangle}(x) -
    \varphi_{\vec{f}^i {^\frown} \langle f_x \rangle}(x) \right)
  \]
  Letting $\vec{g} = \vec{f} ^\frown \langle f_x \rangle$, we have
  \[
    \sum_{i=0}^{n}(-1)^i\varphi_{\vec{g}^i}(x) =\sum_{i=0}^{n}(-1)^i\psi_{\vec{g}^i}(x)
  \]
  Rearranging the terms in this equation yields
  \[
  \sum_{i=0}^{n-1}
  (-1)^i \left(\psi_{\vec{f}^i {^\frown} \langle f_x \rangle}(x) -
  \varphi_{\vec{f}^i {^\frown} \langle f_x \rangle}(x) \right) =
  (-1)^n \left( \varphi_{\vec{f}}(x) - \psi_{\vec{f}}(x)\right)
  \]
  Recall that we chose $x$ so that $\psi_{\vec{f}}(x) = 0$. Hence, putting this all together,
  we obtain
  \[
    \sum_{i=0}^{n-1}(-1)^i \tau_{\vec{f}^i}(x) = (-1)^{2n} \varphi_{\vec{f}}(x)
    = \varphi_{\vec{f}}(x)
  \]
  Since the support of $\psi_{\vec{f}}$ is finite, it follows that
  \[
    \sum_{i=0}^{n-1}(-1)^i \tau_{\vec{f}^i} =^* \varphi_{\vec{f}}
  \]
  as required.
\end{proof}

With the aid of one further lemma, we are now in a position to describe the basic strategy of the argument of our main theorem. When $F$ is a subset of ${^\omega}\omega$
and $\Phi = \left\langle \varphi_{\vec{f}} ~\middle|~ \vec{f} \in ({^\omega} \omega)^n\right\rangle$,
we will write $\Phi\!\restriction\! F$ for
$\left\langle\varphi_{\vec{f}} ~\middle|~ \vec{f}\in F^n\right\rangle$. Recall that a subset $F \subseteq {^\omega} \omega$
is said to be \emph{$\leq^*$-cofinal} if, for all $g \in {^\omega}\omega$, there is
$f \in F$ such that $g \leq^* f$.

\begin{lemma}\label{coftriv} Let $\Phi=\left\langle \varphi_{\vec{f}} ~\middle|~ \vec{f}\in(^\omega\omega)^n\right\rangle$ be an $n$-coherent family of functions. Then $\Phi$ is $n$-trivial if and only if $\Phi\!\restriction\! F$ is $n$-trivial for some $\leq^*$-cofinal $F\subseteq\,\!^\omega\omega$.
\end{lemma}
\begin{proof} The rightwards implication is obvious. For the leftwards implication, the point is that any $n$-trivialization of any such $\Phi\!\restriction\! F$ extends to an $n$-trivialization of $\Phi$. Suppose first that $n = 1$, $F$ is $\leq^*$-cofinal in ${^\omega}\omega$, and
$\psi:\omega^2 \rightarrow \bb{Z}$ trivializes $\Phi\!\restriction\! F$. We claim that $\psi$
trivializes the entire family $\Phi$. Indeed, if $g \in {^\omega}\omega$, then
we can find $f \in F$ such that $g \leq^* f$, i.e., $I(f)$ contains $I(g)$ modulo a
finite set. By the coherence of $\Phi$, we have $\varphi_f =^* \varphi_g$. Since
$\psi$ trivializes $\Phi\!\restriction\! F$, we have $\psi =^* \varphi_f$, and hence
$\psi =^* \varphi_g$, as well.

We can therefore suppose that $n>1$. Let $F$ be $\leq^*$-cofinal in $^\omega\omega$ and let $\Upsilon=\left\langle\upsilon_{\vec{f}} ~\middle|~ \vec{f}\in F^{n-1}\right\rangle$ be an $n$-trivialization of $\Phi\restriction F$. Fix a map $a:\,\!^\omega\omega\to F$ such that $g\leq^* a(g)$ for each $g\in\,\!^\omega\omega$. Write $a(\vec{g})$ for $\left\langle a(g_0),\dots,a(g_{n-1})\right\rangle$. The family $\Psi=\left\langle\psi_{\vec{g}} ~ \middle| ~ \vec{g}\in (^\omega\omega)^{n-1}\right\rangle$ is then an $n$-trivialization of $\Phi$, where
$$\psi_{\vec{g}}=\upsilon_{a(\vec{g})}-\sum_{i=0}^{n-1}(-1)^i\varphi_{\langle g_0,\dots,g_i,a(g_i),\dots,a(g_{n-1})\rangle}$$ for each $\vec{g}\in(^\omega\omega)^{n-1}$. The right-hand side of the equation is a sum of functions defined on all but finitely many elements of $I(\wedge\vec{g})$; the function $\psi_{\vec{g}}$ is defined more precisely as the extension of that sum to the domain $I(\wedge\vec{g})$ by letting $\psi_{\vec{g}}(x)=0$ on any otherwise undefined arguments $x$. The verification that this definition works takes the following shape: for any $\vec{g}=\langle g_0,\dots,g_n\rangle \in (^\omega\omega)^{n+1}$ the terms $(-1)^i\upsilon_{a(\vec{g}^i)}$ $(i\leq n)$ in
\begin{equation}\label{checkingextension}
\sum_{i=0}^n(-1)^i\psi_{\vec{g}^i}
\end{equation}
together simplify (mod finite) to $\varphi_{a(\vec{g})}$. This together with $(n-1)$ of the terms in the sum (\ref{checkingextension}) will simplify (mod finite) to $\varphi_{\langle g_0,a(g_1),\dots,a(g_{n-1})\rangle}$. This process continues until (\ref{checkingextension}) has entirely simplified (mod finite) to the term $\varphi_{\vec{g}}$, thus completing the verification.
\end{proof}
There are in fact a number of ways to argue Lemma \ref{coftriv}. The above approach may be clearer in the following example.
\begin{example} Let $\Phi$ be $2$-coherent. Let $F$ be as in the statement of Lemma \ref{coftriv} and let $\Upsilon=\left\langle\upsilon_f ~\middle|~ f\in F\right\rangle$ $2$-trivialize $\Phi\restriction F$. Let then $a:\,\!^\omega\omega\to F$ be as in the proof of Lemma \ref{coftriv} and let $\psi_f=\upsilon_{a(f)}-\varphi_{\langle f,a(f)\rangle}$ for each $f\in\,\!^\omega\omega$. Then for all $f,g\in\,\!^\omega\omega$,
\begin{align*}
\psi_g-\psi_f & = \,\upsilon_{a(g)}-\upsilon_{a(f)}-(\varphi_{\langle g,a(g) \rangle}-\varphi_{\langle f,a(f)\rangle})\\
&=^* \varphi_{\langle a(f),a(g)\rangle}-\varphi_{\langle g,a(g)\rangle}+\varphi_{\langle f,a(f)\rangle}\\
&=^*\varphi_{\langle f,a(g)\rangle}-\varphi_{\langle g,a(g) \rangle}\\
&=^*\varphi_{\langle f,g \rangle}
\end{align*}
as desired.
\end{example}
We may now describe the broad outlines of our proof of our main theorem. Fix $n>0$ and a weakly compact cardinal $\kappa$, and let $\mathbb{P}$ be a length-$\kappa$ finite-support iteration of Hechler forcings. To show that $V^{\mathbb{P}}\vDash\textnormal{``}\,\mathrm{lim}^{n}\mathbf{A}=0\textnormal{''}$ we fix an arbitrary $n$-coherent $\Phi$ in $V^{\bb{P}}$ and show that it is $n$-trivial. This we accomplish by showing that $\Phi\restriction F$ is $n$-trivial, where $F$ is a $\leq^*$-cofinal family of Hechler reals. This in turn we achieve by defining a family $\left\langle \psi_{\vec{f}} ~ \middle|~ \vec{f}\in F^n\right\rangle$ of finitely supported functions trivializing $\Phi\restriction F$ in the sense of Lemma \ref{finsup}. These $\psi_{\vec{f}}$ are formed out of the differences among particular functions $\varphi_{\vec{f}}$ in $\Phi$. To select those functions we rely heavily on uniformities among the conditions in the $\mathbb{P}$-generic filter which derive from the weak compactness of $\kappa$ by way of the multidimensional $\Delta$-system Lemma of the following section. Our proof consists essentially in applying this procedure to each $n>0$. As indicated, however, expository considerations lead us to break this proof into three stages: that of $n=1$, that of $n=2$, and that of $n>2$.

\subsection{Homological background} \label{hom_background}

We described above a historical sequence of investigations into the derived limits $\lim^n \mathbf{A}$, the point of departure for any of which is the following computation. For each positive integer $n$ let $X^n$ denote the one-point compactification of an infinite countable sum of copies of the $n$-dimensional open unit ball $B^n=\{\vec{x}\in\mathbb{R}^n\mid \lVert\vec{x}\rVert<1\}$; more colloquially, $X^n$ is an $n$-dimensional ``Hawaiian earring.'' Its strong homology groups are as follows:

\[
\oline{H}_{p}(X^n)=
\begin{cases}
      \mathbb{Z} & p= 0 \\
      \prod_{i\in\omega}\mathbb{Z} & p=n \\
      0 & \textnormal{otherwise.}
   \end{cases}
\]
This is in marked contrast to the singular homology groups of $X^n$ \cite{eda}, \cite{barrattmilnor}, and essentially certifies $\oline{H}_{*}$ as a Steenrod homology theory \cite{milnorsteenrod}. For countably infinite sums of copies of $X^n$, on the other hand,
\begin{equation}\label{sum}
\oline{H}_{p}(\coprod_{i\in\omega} X^n)=
\begin{cases}
      \lim^n\mathbf{A}\oplus\big(\!\bigoplus_{i\in\omega}\,\mathbb{Z}\big) & p= 0 \\
      \lim^{n-p}\mathbf{A} & 0<p\leq n \\
      0 & \textnormal{otherwise.}
   \end{cases}
\end{equation}
We will define the inverse system $\mathbf{A}$ and its higher derived limits momentarily; we will then leave all more direct references to strong homology behind until our conclusion. Interested readers are referred to \cite{mp} or to \cite{SSH} more generally for details of the above computations. The significance of these computations, of course, is that they tell us that if strong homology is additive on closed subsets of Euclidean space then $\lim^n\mathbf{A}=0$ for all $n>0$.

Let $\mathcal{N}$ denote the partial order $(^\omega \omega,\leq)$, and let $\mathcal{N}^{\mathrm{op}}$ denote its order-reversal. Let $\tau(\mathcal{N})$ denote the topology on $^\omega\omega$ generated by the basis $\mathsf{B}(\mathcal{N})=\{U(g)\mid g\in\,\! ^\omega\omega\}$, where we recall that $U(g)=\{f\in\,\!^\omega\omega\mid f\leq g\}$. Given a quotient $B/A$ of groups, we write $[b]$ for the coset of an element $b$ of $B$.
\begin{definition}
Let $\mathbf{A}$ denote the inverse system $(A_f,p_{fg},\mathcal{N})$, where $$A_f=\bigoplus_{i\in I(f)}\,\mathbb{Z}$$ and $p_{fg}:A_g\to A_f$ is the projection map for each $f\leq g$ in $\mathcal{N}$. Similarly, let $\mathbf{B}=(B_f,q_{fg},\mathcal{N})$ and $\mathbf{B}/\mathbf{A}=(B_f/A_f,r_{fg},\mathcal{N})$, where $$B_f=\prod_{i\in I(f)}\mathbb{Z}$$ and $q_{fg}:B_g\to B_f$ and $r_{fg}:B_g/A_g\to B_f/A_f$ are projection maps for each $f\leq g$ in $\mathcal{N}$.
\end{definition}
$\mathbf{A}$
is an object, in other words, of the functor or presheaf category $\mathsf{Ab}^{\mathcal{N}^{\mathrm{op}}}$, and it is in this setting that its higher derived limits are most naturally defined. The inverse limit of $\mathbf{A}$, for example, admits both a category-theoretic characterization via a universal property in $\mathsf{Ab}^{\mathcal{N}^{\mathrm{op}}}$ and the following more concrete characterization:
\[
  \lim\,\mathbf{A}=\left\{\langle\varphi_f \mid f\in\,\!^\omega\omega\rangle\in\prod_{f\in\,\!^\omega\omega}A_f ~ \middle| ~ p_{fg}(\varphi_g)=\varphi_f\textnormal{ for all }f\leq g\textnormal{ in }\,\!^\omega\omega\right\}
\]
More precisely, $\lim\,\mathbf{A}$ is the above family naturally viewed as a group.\footnote{Observe that in consequence $\lim^0\mathbf{A}$, which equals $\lim\mathbf{A}$ by definition, is isomorphic to $\bigoplus_{i\in\omega}\prod_{j\in\omega}\mathbb{Z}$, as additivity in the degree $n=p$ in equation (\ref{sum}) would require.} Similarly, higher derived limits $\lim^n\mathbf{A}$ admit abelian category theoretic description as the cohomology groups of the $\lim$-image of an injective resolution of $\mathbf{A}$. Standard resolutions convert this description to a more concrete general form, as above. The characterizations of $\lim^1\mathbf{A}$ and of $\lim^n\mathbf{A}$ $(n\geq 1)$ in \cite{mp} and \cite{b}, respectively, then each involved one further conversion, via the long exact sequence
\begin{align}\label{les}0\longrightarrow \,\text{lim}\,\mathbf{A}\longrightarrow \text{lim}\,\mathbf{B}\longrightarrow \text{lim}\,\mathbf{B}/\mathbf{A}
\xrightarrow{\partial^0} \text{lim}^1\,\mathbf{A}\longrightarrow \text{lim}^1\,\mathbf{B}\longrightarrow \cdots\end{align}
associated to the natural short exact sequence
\begin{align*}0\longrightarrow \,\mathbf{A}\longrightarrow \,\mathbf{B}\longrightarrow \,\mathbf{B}/\mathbf{A}\longrightarrow 0.
\end{align*}
We will require for our work below two further reformulations.

First, for each $U$ in $\tau(\mathcal{N})$ let
\[
  \lim(\mathbf{A}\restriction U)=\left\{\langle\varphi_f \mid f\in U\rangle\in\prod_{f\in U}
  A_f ~ \middle| ~ p_{fg}(\varphi_g)=\varphi_f\textnormal{ for all }f\leq g\textnormal{ in }U\right\}.
\]
Observe that \begin{align}\label{trivlim}\lim(\mathbf{A}\restriction U(g))\cong A_g\textnormal{  for any }g\in\,\!^\omega\omega.\end{align} Similar definitions and observations apply for the systems $\mathbf{B}$ and $\mathbf{B}/\mathbf{A}$. By Jensen's observation in \cite[page 4]{jensen},
$$\mathrm{lim}^n\mathbf{A}\cong\check{H}^n\big((^\omega\omega,\tau(\mathcal{N})),\mathcal{A}\big)$$
where $\mathcal{A}$ is the sheaf on $(^\omega\omega,\tau(\mathcal{N}))$ given by $U\mapsto \lim(\mathbf{A}\restriction U)$ together with the natural restriction maps. Analogous definitions apply for the systems $\mathbf{B}$ and $\mathbf{B}/\mathbf{A}$; in particular, $\mathcal{B}/\mathcal{A}$ will denote the sheaf given by $U\mapsto \lim(\mathbf{B}/\mathbf{A}\restriction U)$. Since $\mathsf{B}(\mathcal{N})$ is maximal in the refinement-ordering of open covers of $\mathcal{N}$, equation (\ref{les}) now takes the following form:
\begin{align*}\label{les2}\dots\rightarrow \check{H}^0(\mathsf{B}(\mathcal{N}),\mathcal{B})\rightarrow \check{H}^0(\mathsf{B}(\mathcal{N}),\mathcal{B}/\mathcal{A})
\xrightarrow{\partial^0} \check{H}^1(\mathsf{B}(\mathcal{N}),\mathcal{A})\rightarrow \check{H}^1(\mathsf{B}(\mathcal{N}),\mathcal{B})\rightarrow \cdots\end{align*}
As observed in \cite{b}, \begin{equation}\label{B}\mathrm{lim}^n\,\mathbf{B}=0\textnormal{ for all }n>0.\end{equation}
The same therefore holds for $\check{H}^n(\mathsf{B}(\mathcal{N}),\mathcal{B})$. In consequence,\begin{equation}\label{lim1a}\mathrm{lim}^1\mathbf{A}\cong\frac{\check{H}^0(\mathsf{B}(\mathcal{N}),\mathcal{B}/\mathcal{A})}{\mathrm{im}(\check{H}^0(\mathsf{B}(\mathcal{N}),\mathcal{B}))}.\end{equation}
We will define these cohomology groups more explicitly below. It will then be clear that the natural rendering of the above quotient is as the \emph{coherent families of functions modulo the trivial families of functions} in the sense of Definition \ref{cohtriv}.
The following is then immediate:
\begin{theorem}[Theorem 1 in \cite{mp}]\label{lim1athm} $\lim^1\mathbf{A}=0$ if and only if every coherent family of functions $\left\langle\varphi_f ~\middle|~ f \in {^\omega}\omega \right\rangle$ is trivial.
\end{theorem}
By point (\ref{B}) above, the higher-$n$ analogue of equation (\ref{lim1a}) is simply
\begin{equation}\label{limna}\mathrm{lim}^n\mathbf{A}\cong\check{H}^{n-1}(\mathsf{B}(\mathcal{N}),\mathcal{B}/\mathcal{A})\textnormal{ for all }n>1\end{equation}
The advantage of this formulation is that in computing \v{C}ech cohomology groups we have a choice among several indexing schema giving equivalent quotients \cite[I.20 Proposition 2]{serre}. As a generic real may fail to integrate nontrivially into the ordering $({^\omega}\omega,\leq)$ of its ground model, it is useful in contexts like ours to define higher coherence via indices as untethered to that order as possible. To that end we adopt the \emph{alternating cochains} definition of the \v{C}ech complex.
\begin{definition}\label{Cech} For $n\geq 0$, the group $\check{H}^n(\mathsf{B}(\mathcal{N}),\mathcal{B}/\mathcal{A})$ is the cohomology of the cochain complex
\begin{align}\label{cechcomplex}\cdots\xrightarrow{d^{n-1}} C_{alt}^n(\mathsf{B}(\mathcal{N}),\mathcal{B}/\mathcal{A})\xrightarrow{d^n}C_{alt}^{n+1}(\mathsf{B}(\mathcal{N}),\mathcal{B}/\mathcal{A})\xrightarrow{d^{n+1}}\cdots\end{align}
where $C_{alt}^n(\mathsf{B}(\mathcal{N}),\mathcal{B}/\mathcal{A})$ denotes the subgroup of
\begin{align}\label{altcochains}
\prod_{\vec{f}\in(^\omega \omega)^{n+1}}\mathrm{lim}\left(\mathbf{B}/\mathbf{A}\restriction U(\wedge\vec{f})\right)=\prod_{\vec{f}\in(^\omega \omega)^{n+1}}B_{\bigwedge\vec{f}}\,/A_{\bigwedge\vec{f}}
\end{align}
whose elements $c$ satisfy $$c(\sigma(\vec{f}))=sgn(\sigma)c(\vec{f})$$
for all $\vec{f}\in(^\omega\omega)^{n+1}$ and permutations $\sigma$ of $\langle 0,\dots,n\rangle$. The equality in  (\ref{altcochains}) amounts simply to an application of the $\mathbf{B}/\mathbf{A}$-variant of observation (\ref{trivlim}).

The differentials $d^n$ are defined as usual. Namely, for any $i\leq n+1$ and $\vec{f}\in(^\omega \omega)^{n+2}$ recall that $\vec{f}^i$ denotes the element of $(^\omega\omega)^{n+1}$ formed via the omission of the coordinate $f_i$ from $\vec{f}$. Then for any $c\in C_{alt}^n(\mathsf{B}(\mathcal{N}),\mathcal{B}/\mathcal{A})$, define $d^n c$ by the assignment
\begin{align}\label{sum1}
d^n c(\vec{f})=\displaystyle\sum_{i=0}^{n+1}(-1)^i \left(c(\vec{f}^i)\restriction I(\wedge \vec{f})\right)
\end{align}
for each $\vec{f}\in (^\omega\omega)^{n+2}$.

Entirely analogous definitions hold for $\check{H}^n(\mathsf{B}(\mathcal{N}),\mathcal{A})$ and $\check{H}^n(\mathsf{B}(\mathcal{N}),\mathcal{B})$. Observe in particular that any $c$ in $C_{alt}^n(\mathsf{B}(\mathcal{N}),\mathcal{B}/\mathcal{A})$ admits representatives $\hat{c}$ in $C_{alt}^n(\mathsf{B}(\mathcal{N}),\mathcal{B})$, i.e., alternating cochains $\hat{c}$ taking values in the groups $B_{\bigwedge\vec{f}}$ and satisfying $[\hat{c}(\vec{f})]=c(\vec{f})$ for all $\vec{f}\in(^\omega\omega)^{n+1}$.
\end{definition}
As in the cases of Section \ref{set_background}, the requisite function-restrictions will tend to be clear from context. Hence again for readability we will rarely notate them, writing,
for example, sums like that of equation (\ref{sum1}) as $$\displaystyle\sum_{i=0}^{n+1}(-1)^i c(\vec{f}^i).$$

It is now clear that $n$-coherent and $n$-trivial families are, respectively,
representatives of the $(n-1)$-cocycles and $(n-1)$-coboundaries of the complex (\ref{cechcomplex}) above.
Recall that every $n$-trivial $\Phi$ is $n$-coherent. We have just argued that
whether or not $\lim^n\mathbf{A}$ vanishes is precisely the question of the converse;
we have in other words shown the following:
\begin{theorem}\label{limnA} Let $n$ be a positive integer. Then $\text{lim}^n\,\mathbf{A}=0$ if and only if every $n$-coherent family
$\Phi=\left\langle\varphi_{\vec{f}} ~\middle|~ \vec{f}\in(^\omega\omega)^n\right\rangle$ is $n$-trivial.
\end{theorem}

We end this subsection by indicating how Lemma \ref{finsup} can be derived
from cohomological considerations in the special case $F = {^\omega}\omega$.
(The more general case of an arbitrary $F \subseteq {^\omega}\omega$ can be obtained
via a routine modification of the arguments in this subsection, working with the
partial order $\mathcal{N}_F = (F, \leq)$ and inverse systems
$\mathbf{A} \restriction F = (A_f, p_{fg}, \mathcal{N}_F)$, $\mathbf{B} \restriction F$,
and ${\mathbf{B} / \mathbf{A}} \restriction F$ instead of $\mathcal{N}$, $\mathbf{A}$,
$\mathbf{B}$, and $\mathbf{B}/\mathbf{A}$. All of the results of this subsection carry through with
no difficulty in this setting.)
We return to equation (\ref{limna}), which had derived from the cohomological reformulation of equation (\ref{les}). Equation (\ref{limna}) is more precisely stated as follows:
\begin{lemma}\label{connectingmap} Let $n$ be greater than $1$. Then $\check{H}^{n-1}(\mathsf{B}(\mathcal{N}),\mathcal{B}/\mathcal{A})$ is isomorphic to $\check{H}^n(\mathsf{B}(\mathcal{N}),\mathcal{A})$ via the connecting map $\partial_{n-1}: [c]\mapsto [d^{n-1} \hat{c}]$, where $\hat{c}$ is any representative of $c$ in $C_{alt}^{n-1}(\mathsf{B}(\mathcal{N}),\mathcal{B})$.
\end{lemma}
\begin{proof}
Here the only new assertion is the characterization of the connecting map $\partial_{n-1}$. This, though, is simply the general form of the connecting maps of a long exact sequence; see for example \cite[Addendum 1.3.3]{weibel}.
\end{proof}
For $n\geq 1$ it follows that an $n$-coherent family $\Phi=\left\langle\varphi_{\vec{f}}
~\middle|~ \vec{f}\in(^\omega\omega)^n\right\rangle$ represents a coboundary in $C_{alt}^n(\mathsf{B}(\mathcal{N}),\mathcal{B}/\mathcal{A})$ if and only if what we might write as $d^n\Phi$ \emph{is} a coboundary in $C_{alt}^{n+1}(\mathsf{B}(\mathcal{N}),\mathcal{A})$.
Lemma \ref{finsup} is then immediate.

\section{A multidimensional $\Delta$-system lemma}\label{multi}

We turn now more towards the argument of our main theorem. We begin by proving a generalization of the two-dimensional $\Delta$-system lemma appearing in \cite{bhs} (see \cite{TodReals} for an earlier precedent). Structuring our generalization will be the following refinement of the relation holding between pairs of elements of a classical $\Delta$-system.
\begin{definition}
  Sets $u,v \in [\ord]^{<\omega}$ are \emph{aligned} if
  \begin{itemize}
    \item $|u| = |v|$, and
    \item $|u \cap \alpha| =
    |v \cap \alpha|$ for all $\alpha \in u \cap v$.
  \end{itemize}
\end{definition}

\begin{definition}
Let $0 < k < \ell$ be natural numbers. A function $t:\ell \rightarrow 3$
  is a \emph{type of width $k$ and length $\ell$} if
  \[
    |t^{-1}[\{0,2\}]| = |t^{-1}[\{1,2\}]| = k.
  \]
  Given $u,v \in [\ord]^k$ and an increasing enumeration $\langle\alpha_i \mid i < \ell\rangle$ of $u \cup v$, we say that the \emph{type of $u$ and $v$} is the unique type $t$ of width $k$ and length $\ell$ such that
  \begin{itemize}
    \item $u \setminus v = \{\alpha_i \mid t(i)=0\}$,
    \item $v \setminus u = \{\alpha_i \mid t(i)=1\}$, and
    \item $u \cap v = \{\alpha_i \mid t(i)=2\}$.
  \end{itemize}
 We denote this unique type $t$ by
  $\tp(u,v)$. Observe that $u$ and $v$ are aligned if and only if $t=\tp(u,v)$ satisfies
  \[
  |t^{-1}[\{0\}] \cap i| = |t^{-1}[\{1\}] \cap i|
\]
for all $i $ in $t^{-1}[\{2\}]$. Any such type $t$ is also called \emph{aligned}.
\end{definition}

Types as above admit natural representations as finite strings of $0$s, $1$s, and $2$s. If $\alpha<\beta<\gamma$, for example, then
  $\tp(\{\alpha,\beta\},\{\beta,\gamma\})$ is naturally rendered as $021$.

We may now state our multidimensional $\Delta$-system lemma. We first remind the reader
of the definition of a \emph{weakly compact cardinal}.

\begin{definition}
  A cardinal $\kappa$ is \emph{weakly compact} if $\kappa$ is uncountable and,
  for every positive integer $n$, every $\theta < \kappa$, and every function
  $c:[\kappa]^n \rightarrow \theta$, there is a set $H \in [\kappa]^\kappa$
  such that $c \restriction [H]^n$ is constant.
\end{definition}

There are many useful alternative characterizations of weakly compact cardinals;
for example, $\kappa$ is weakly compact if and only if $\kappa$ is strongly
inaccessible and has the tree property.
We refer the reader to \cite{kanamori} for more information on this and other
equivalent formulations of weak compactness.

Recall our notational convention that elements of $[\kappa]^n$ are often thought
of as increasing $n$-tuples. So, for instance, when we write $\vec{\alpha} =
\langle \alpha_0, \ldots, \alpha_{n-1} \rangle \in [\kappa]^n$, it is implicit
that $\alpha_0 < \alpha_1 < \ldots < \alpha_{n-1}$.

\begin{lemma} \label{strong_delta_system_lemma}
  Let $n$ be a positive number and let $\kappa$ be a weakly compact cardinal, and suppose that the family of $n$-tuples $\vec{\alpha} = \langle \alpha_0, \ldots, \alpha_{n-1} \rangle\in [\kappa]^n$ indexes elements $u_{\vec{\alpha}}$ of $ [\kappa]^{<\omega}$.
  Then there is a set $A \in [\kappa]^\kappa$ and a collection
  $\left\langle u_{\vec{\alpha}} ~ \middle| ~ \vec{\alpha}\in [A]^{<n}\right\rangle$
  of elements of $[\kappa]^{<\omega}$ such that
  \begin{enumerate}
    \item for all $\vec{\alpha} \in [A]^{<n}$,
    \begin{enumerate}
      \item if $\beta \in A$ and $\vec{\alpha} < \beta$, then
      $u_{\vec{\alpha}} < \beta$,
      \item if $\vec{\beta} \in [A]^{\leq n}$ satisfies
      $\vec{\alpha} \sqsubseteq \vec{\beta}$, then $u_{\vec{\alpha}} \sqsubseteq
      u_{\vec{\beta}}$,
      \item the set $\left\{u_{\vec{\alpha}
      ^\frown \langle \beta \rangle} ~ \middle| ~ \beta \in A\backslash(\max(\vec{\alpha})+1)
      \right\}$ forms a $\Delta$-system with root $u_{\vec{\alpha}}$;
    \end{enumerate}
    \item for all $m \leq n$ and all $\vec{\alpha}, \vec{\beta} \in [A]^m$,
    \begin{enumerate}
    \item $|u_{\vec{\alpha}}| = |u_{\vec{\beta}}|$, and
    \item if $\vec{\alpha}$ and $\vec{\beta}$ are aligned, then $u_{\vec{\alpha}}$
    and $u_{\vec{\beta}}$ are aligned.
    \end{enumerate}
  \end{enumerate}
\end{lemma}
To prove Lemma \ref{strong_delta_system_lemma}, we require one further technical lemma.
\begin{lemma} \label{type_cycle_lemma}
  Suppose that $0 < k < \ell < \omega$ and that $t$ is an aligned type of width
  $k$ and length $\ell$. Then for some positive number $m$ there exists a sequence
  $\langle u_j \mid j \leq 2m \rangle$ of elements of $[\ord]^k$ such that
  \begin{itemize}
    \item $\tp(u_{2j}, u_{2j+1}) = \tp(u_{2j + 2}, u_{2j + 1}) = t$ for all $j < m$, and
    \item $\tp(u_0, u_{2m}) = t$.
  \end{itemize}
\end{lemma}

\begin{proof}
  The proof is by induction on $k$. We will also assume that $t(\ell - 1) \in \{1,2\}$.
  The case in which $t(\ell - 1) = 0$
  is essentially the same as that in which $t(\ell - 1) = 1$ and is left to the
  reader. If $k = 1$, then either $t = 2$ or $t = 01$. We may take $m = 1$ in either case. For if $t=2$ then for any ordinal $\alpha$ the sequence
  $u_0 = u_1 = u_2 = \{\alpha\}$ is as desired, while if $t=01$ then for any ordinals $\alpha < \beta < \gamma$ the sequence $u_0 = \{\alpha\}$, $u_1 = \{\gamma\}$, $u_2 = \{\beta\}$  is as desired.

  Now assume that $k > 1$ and that we have established the lemma for all aligned
  types of width less than $k$.

  \textbf{Case 1: $t(\ell - 1) = 2$.} In this case, $t' = t \restriction (\ell - 1)$
  is an aligned type of width $k - 1$, hence by our inductive hypothesis
  there exists some positive $m$ and sequence $\langle u_j \mid j \leq 2m \rangle$
  of elements of $[\ord]^{k-1}$ witnessing the conclusion of the lemma for $t'$.
  Now take any $\beta>\bigcup_{j \leq 2m}u_j$, and
  let $u_j^\star = u_j \cup \{\beta\}$ for all $j$. Then $m$ and $\langle u_j^\star \mid
  j \leq 2m \rangle$ witness the conclusion of the lemma for $t$.

  \textbf{Case 2: $t(\ell - 1) = 1$.} The idea in this case is to  form a type $t'$ of
  width $k - 1$ and length $\ell-2$ by deleting the last $0$ and the last $1$ from $t$. The induction hypothesis then furnishes a sequence $\langle u_j\,|\,j\leq 2m\rangle$ as in the conclusion of the lemma for the type $t'$. It is then fairly straightforward (though notationally tedious) to modify $\langle u_j\,|\,j\leq 2m\rangle$ to form a sequence as desired for the original type $t$.

  More formally, let $\ell_0 < \ell$ be such that $t(\ell_0) = 0$ and $t(i) = 1$ for all $i$ with $\ell_0 < i < \ell$. Such an $\ell_0$ must exist
  since $t$ is aligned. Let $k_0 = |t^{-1}[\{1,2\}] \cap \ell_0|$.
  This number $k_0$ has the following significance: if $u,v\in[\ord]^k$ are such that $\tp(u,v) = t$, then $k_0$ is the number of elements of $v$ that are smaller than
  the largest element of $u$. In other words, $v(k_0)$ is the least element of $v$ larger
  than every element of $u$. Now define $t' : (\ell - 2) \rightarrow 3$
  by letting
  \[
    t'(i) =
    \begin{cases}
      t(i) & \text{if } i < \ell_0 \\
      1 & \text{if } \ell_0 \leq i < \ell - 2
    \end{cases}
  \]
This defines an aligned type $t'$ of width $k - 1$. 
  Apply the induction hypothesis to obtain a positive $m$ and a
  sequence $\langle u_j \mid j \leq 2m \rangle$ of elements of $[\ord]^{k-1}$
  witnessing the conclusion of the lemma for $t'$. By ``stretching" each $u_j$
  if necessary, we may assume that every element of every $u_j$ is a limit of limit ordinals.

  We now augment the elements of the sequence $\langle u_j \mid j \leq 2m \rangle$ to form a sequence $\langle u^\star_j \mid j \leq 2n \rangle$ for $t$. Two cases are straightforward: those in which $k_0=0$ and those in which $k_0=k-1$. In the latter case, for example, take $\{\beta_j\,|\,j\leq 2m\}>\bigcup_{j\leq 2m}u_j$ such that
  \begin{itemize}
    \item for all $j, j' \leq 2m$, if $j$ is even and $j'$ is odd, then $\beta_j < \beta_{j'}$;
    \item $\beta_0 < \beta_{2m}$.
  \end{itemize}
 Letting $u_j^\star=u_j\cup\{\beta_j\}$ then defines a sequence as desired for $t$. The case of $k_0=0$ is similar. Therefore assume that $0<k_0<k-1$. We will assume as well that $u_j(k_0) > u_{j'}(k_0 - 1)$ for all odd $j,j'<2m$. We are free to assume this because if $\tp(u,v) = t'$ then $v(k_0)$ is larger than every element of $u$, hence
  $\tp(u,w) = t'$ for any $w$ with $w\!\restriction\! k_0 = v\! \restriction\! k_0$
  and $w\! \restriction\! [k_0, k-1)$ consisting of arbitrarily large ordinals.

 We now define limit ordinals $\langle \beta_j \mid j \leq 2m \rangle$ so that $u_j < \beta_j$ for
  each $j \leq 2m$. We first do this for the even $j \leq 2m$. First let $\beta_0$ be a limit ordinal such that $u_1(k_0-1)<\beta_0<u_1(k_0)$, and let $\beta_{2m}$ be a limit ordinal such that $u_{2m-1}(k_0-1)<\beta_{2m}<u_{2m-1}(k_0)$. For any other even $0<j<2m$ let $\beta_j>u_j$ be a limit ordinal satisfying
  \[
    \max\{u_{j-1}(k_0 - 1), u_{j+1}(k_0 - 1)\}<\beta_j<\min\{u_{j-1}(k_0),
    u_{j+1}(k_0)\}.
  \]
  This is all possible by the assumptions of the preceding paragraphs.

   Lastly, let $\beta^\star$ be a limit ordinal larger than any ordinal so far considered, and for all odd $j < 2m$ let $\beta_j = \beta^\star$. (More precisely we require that $\beta^\star > \beta_j$ for all even $j \leq 2m$
  and $\beta^\star > u_j$ for all odd $j < 2m$.) Now let $u^\star_j = u_j \cup \{\beta_j\}$ for each $j \leq 2m$. By construction, $\tp(u^\star_{2j},
  u^\star_{2j+1}) = \tp(u^\star_{2j+2}, u^\star_{2j + 1}) = t$ for all $j < m$.
  However, $\tp(u^\star_0, u^\star_{2m}) = t$ may fail to hold: although $\tp(u_0, u_{2m})$ indeed equals $t'$, it may not be the case that $u^\star_{2m}(k_0 - 1)<u^\star_0(k-1) < u^\star_{2m}(k_0)$, as $\tp(u^\star_0, u^\star_{2m}) = t$ would require.

  We address this possibility by adding a $u^\star_{2m + 1}$ and $u^\star_{2m + 2}$ to our sequence and
  then ``rotating'' it two steps.  First, let $u^\star_{2m + 2}$ be any element of $[\ord]^k$ such that
  \begin{itemize}
    \item $\tp(u^\star_0, u^\star_{2m + 2}) = t$,
    \item if $k_0 > 0$, then $u^\star_{2m + 2}(k_0 - 1) < u^\star_{2m}(k_0)$
    (this is possible since $u^\star_0(k-2) < u^\star_{2m}(k_0)$ and the latter
    is a limit ordinal), and
    \item $u^\star_{2m + 2}(k_0) > \beta^\star$.
  \end{itemize}
  Now define $u^\star_{2m + 1}$ by letting $u^\star_{2m + 1} \restriction (k-1) =
  u^\star_0 \restriction (k-1)$ and letting $u^\star_{2m + 1}(k-1)$ be any ordinal
  $\gamma > u^\star_0(k-2)$ such that
  \[
    \max\{u^\star_{2m + 2}(k_0 -1), u^*_{2m}(k_0 - 1)\}<\gamma < u^\star_{2m}(k_0)
  \]
  (and hence $\gamma < u^\star_{2m + 2}(k_0)$). Then
  \[
    \tp(u^\star_{2m + 1}, u^\star_{2m}) = \tp(u^\star_{2m + 1}, u^*_{2m + 2}) =
    \tp(u^\star_0, u^\star_{2m + 2}) = t.
  \]
  Therefore, if we define $\langle v_j \mid j \leq 2m + 2 \rangle$
  by setting $v_0 = u^\star_{2m + 1}$, setting $v_1 = u^\star_{2m + 2}$ and, for $2 \leq j < 2m + 2$,
  setting $v_j = u^\star_{j - 2}$, we see that $\langle v_j \mid j \leq 2m + 2 \rangle$
  witnesses the conclusion of the theorem for $t$.
\end{proof}
Readers may better appreciate the necessity of some argument like the above after constructing a cycle as in Lemma \ref{type_cycle_lemma} for the type $001011$. For yet more on these matters, readers are referred to \cite{types}.

We turn now to the proof of our main lemma:

\begin{proof}[Proof of Lemma \ref{strong_delta_system_lemma}]
  By the weak compactness of $\kappa$ there exists a natural number $\ell$ and a $B\in [\kappa]^\kappa$ such that $|u_{\vec{\alpha}}| = \ell$ for all $\vec{\alpha} \in [B]^{n}$.

  For any $\vec{\gamma} = \langle \gamma_0, \ldots, \gamma_{2n-1} \rangle
  \in [B]^{2n}$ and $I \in [2n]^n$ let $\vec{\gamma}[I] =
  \langle \gamma_i \mid i \in I \rangle$,  and let
  \[
    U_{\vec{\gamma}} = \left\{\gamma_0, \ldots, \gamma_{2n-1}\right\} \ \cup \bigcup_{I \in [2n]^n}
    u_{\vec{\gamma}[I]}.
  \]
  Recall that $H(\omega)$ denotes the set of all hereditarily finite sets (in particular,
  $H(\omega)$ is itself a countable set). Define a coloring $c:[B]^{2n} \rightarrow H(\omega)$ by letting
  \[
    c(\vec{\gamma}) = \big\langle \otp(U_{\vec{\gamma}}), ~ \langle \xi_i \mid i < 2n \rangle, ~
    \langle v_I \mid I \in [2n]^n \rangle\big\rangle,
  \]
  where the latter sequences record the images of the $\gamma_i$ and $u_{\vec{\gamma}[I]}$, respectively, under the order-isomorphism $h:U_{\vec{\gamma}}\to \otp(U_{\vec{\gamma}})$; more precisely,
  \begin{itemize}
    \item $\xi_i = h(\gamma_i)$ for each $i < 2n$,
    \item $v_I = h[u_{\vec{\gamma}[I]}]$ for each $I \in [2n]^n$.
  \end{itemize}

  Again by the weak compactness of $\kappa$ we may fix a set $A \in [B]^\kappa$ such that
  $c$ is constant on $[A]^{2n}$. 

  We now define $u_{\vec{\alpha}}$ for $\vec{\alpha}$ in $[A]^{\leq n}$
  by reverse induction on $|\vec{\alpha}|$. The base case
  $\{ u_{\vec{\alpha}}\,|\,\vec{\alpha} \in [A]^n\}$ is given. In the induction step,
  \begin{itemize}
  \item $0<m<n$,
  \item $u_{\vec{\beta}}$ is defined for all
  $\vec{\beta}$ in $[A]^{\leq n}$ of length greater than $m$, and
  \item the $u_{\vec{\beta}}$ with $\vec{\beta}$ in $ [A]^{m+1}$ are all of the same size.
\end{itemize}
Now fix $\vec{\alpha}$ in $[A]^m$.

  \begin{claim} \label{claim_36}
    The set $\left\{u_{\vec{\alpha} ^\frown \langle \beta \rangle} ~ \middle| ~
    \beta \in A\backslash( \max(\vec{\alpha})+1)\right\}$ forms a $\Delta$-system
    with a root $r$ such that $r < \beta$ and $r \sqsubseteq u_{\vec{\alpha} ^\frown \beta}$
    for every $\beta\in A$ with $\beta > \max(\vec{\alpha})$.
  \end{claim}

  \begin{proof}
    First suppose that $m=n-1$. Consider any $\beta_0 < \beta_1 < \beta_2$
    in $A\backslash( \max(\vec{\alpha})+1)$ together with any $\vec{\gamma}_0, \vec{\gamma}_1, \vec{\gamma}_2$ in $ [A]^{2n}$
  such that
    \begin{itemize}
      \item $\vec{\alpha} ^\frown \langle \beta_1, \beta_2\rangle \sqsubseteq \vec{\gamma_0}$,
      \item $\vec{\alpha} ^\frown \langle \beta_0, \beta_2\rangle \sqsubseteq \vec{\gamma_1}$,
      \item $\vec{\alpha} ^\frown \langle \beta_0, \beta_1\rangle \sqsubseteq \vec{\gamma_2}$.
    \end{itemize}
   Let
    $\vec{\varepsilon}(\beta_j) = \vec{\alpha} ^\frown \langle \beta_j\rangle$ for $j < 3$. Then $c(\vec{\gamma_0}) = c(\vec{\gamma_1}) = c(\vec{\gamma_2})$ implies that
    \[
      \eta \in u_{\vec{\varepsilon}(\beta_0)} \cap u_{\vec{\varepsilon}(\beta_1)} \Leftrightarrow
      \eta \in u_{\vec{\varepsilon}(\beta_0)} \cap u_{\vec{\varepsilon}(\beta_2)} \Leftrightarrow
      \eta \in u_{\vec{\varepsilon}(\beta_1)} \cap u_{\vec{\varepsilon}(\beta_2)}
    \]
    for any $\eta$ in $\kappa$. It follows immediately that $\left\{u_{\vec{\alpha}
    ^\frown \langle \beta \rangle} ~ \middle| ~
    \beta \in A\backslash( \max(\vec{\alpha})+1)\right\}$ forms a $\Delta$-system. Let
    $r$ be its root. Routine pigeonhole arguments together with the homogeneity of $A$ then show that $r < \beta$ and $r \sqsubseteq u_{\vec{\alpha}
    ^\frown \langle \beta \rangle}$ for every $\beta \in A$ with $\beta > \max({\vec{\alpha}})$.

    Now suppose that $m<n-1$. We argue again essentially as above, beginning instead with a sequence $\beta_0 < \beta_1 < \beta_2<\delta_0<\dots<\delta_{n-m-2}$
    in $A\backslash( \max(\vec{\alpha})+1)$. Again take $\vec{\gamma}_0\sqsupseteq\vec{\alpha}^\frown\langle\beta_1,\beta_2,\delta_0,\dots,\delta_{n-m-2}\rangle$; similarly for $\vec{\gamma}_1$ and $\vec{\gamma}_2$. Let
    $\vec{\varepsilon}(\beta_j)=\vec{\alpha}^\frown\langle\beta_j,\delta_0,\dots,\delta_{n-m-2}\rangle$ for each $\beta_j$. Then again the $c$-homogeneity of $A$ implies that $\left\{u_{\vec{\varepsilon}(\beta)} ~ \middle| ~ \beta\in (\alpha_{m-1},\delta_0)\cap A\right\}$
    forms a $\Delta$-system. Since $u_{\vec{\alpha} ^\frown \langle \beta \rangle}\sqsubseteq u_{\vec{\varepsilon}(\beta)}$ for each such $\beta$, and since the sets $u_{\vec{\alpha} ^\frown \langle \beta \rangle}$ are all of the same size,
    $\left\{u_{\vec{\alpha} ^\frown \langle \beta \rangle}~ \middle| ~ \beta\in (\alpha_{m-1},\delta_0)\cap A\right\}$ forms a $\Delta$-system as well. As $\delta_0$ was arbitrary, $\left\{u_{\vec{\alpha} ^\frown \langle \beta \rangle} ~ \middle| ~
    \beta \in A\backslash( \max(\vec{\alpha})+1)\right\}$ forms a $\Delta$-system; pigeonhole and homogeneity arguments, as above, secure the remainder of the claim.
  \end{proof}

  Now let $u_{\vec{\alpha}}$ equal the root of the $\Delta$-system
  $\left\{u_{\vec{\alpha} ^\frown \langle \beta \rangle} ~ \middle| ~
  \beta \in A\backslash( \max(\vec{\alpha})+1)\right\}$. By Claim~\ref{claim_36},
  the homogeneity of $A$, and a pigeonhole argument, we see that we have satisfied item (1) of Lemma \ref{strong_delta_system_lemma}. To see that we satisfy item (2a) as well (and hence that our third inductive assumption was justified), observe that we may calculate $|u_{\vec{\alpha}}|$ from the constant value of $c\!\restriction\![A]^{2n}$, which we denote by $\big\langle \ell^*,
  \langle \xi^*_i \mid i < 2n \rangle, \langle v^*_I \mid I \in [2n]^n \rangle\big\rangle$; the point is that this computation depends only on the size, $m$, of $\vec{\alpha}$. To this end, let $\mathcal{I}_m=\{I\in [2n]^n\,|\,I\cap m=m\}$. Consideration of $c(\vec{\gamma})$
  for any $\vec{\gamma}\in [A]^{2n}$ end-extending $\vec{\alpha}$ then shows that $|u_{\vec{\alpha}}| = |\bigcap_{I\in\mathcal{I}_m}v^*_{I}|$.

  It remains to verify item (2b). To this end, fix $m \leq n$ and an aligned pair of $m$-tuples
  $\vec{\alpha}, \vec{\beta} \subseteq A$. Suppose for the sake of contradiction that
  $u_{\vec{\alpha}}$ and $u_{\vec{\beta}}$ are not aligned and fix
  $\eta \in u_{\vec{\alpha}} \cap u_{\vec{\beta}}$ such that $k_0 =
  |u_{\vec{\alpha}} \cap \eta| \neq k_1 = |u_{\vec{\beta}} \cap \eta|$.
  Let $t = \tp(\vec{\alpha}, \vec{\beta})$.
  Observe that if $\tp(\vec{\gamma}, \vec{\delta}) = t$ for some $\vec{\gamma},
  \vec{\delta} \in [A]^m$
  then $u_{\vec{\gamma}}(k_0) = u_{\vec{\delta}}(k_1)$, by the homogeneity of $A$.

  By Lemma \ref{type_cycle_lemma} there exists an integer $\ell \geq 1$ and
  a sequence $\langle \vec{\gamma}_j \mid j \leq 2\ell \rangle$ of elements of
  $[A]^m$ such that $\tp(\vec{\gamma}_0, \vec{\gamma}_{2\ell}) = t$ and $\tp(\vec{\gamma}_{2j}, \vec{\gamma}_{2j+1}) =
  \tp(\vec{\gamma}_{2j+2}, \vec{\gamma}_{2j+1}) = t$
  for all $j < \ell$. Let
  $\xi = u_{\vec{\gamma}_0}(k_0)$.

  \begin{claim}
     $u_{\vec{\gamma}_{2j}}(k_0) = \xi$ for all $j < m$.
  \end{claim}

  \begin{proof}
    We proceed by induction on $j$. For $j = 0$, the claim is trivial. Suppose that
    $0 < j < m$ and $u_{\vec{\gamma}_{2j - 2}}(k_0) = \xi$. Then, since
    $\tp(\vec{\gamma}_{2j-2}, \vec{\gamma}_{2j-1}) = t = \tp (\vec{\gamma}_{2j},
    \vec{\gamma}_{2j - 1})$, we have that $\xi = u_{\vec{\gamma}_{2j-2}}(k_0) =
    u_{\vec{\gamma}_{2j-1}}(k_1) = u_{\vec{\gamma}_{2j}}(k_0)$, as desired.
  \end{proof}

It follows from the claim that $u_{\vec{\gamma}_{2m}}(k_0) = \xi$. However, since
  $\tp(\vec{\gamma}_0, \vec{\gamma}_{2m}) = t$, we also have
  $u_{\vec{\gamma}_{2m}}(k_1) = \xi$, contradicting the fact that $k_0 \neq k_1$.
  This completes the proof of the lemma.
\end{proof}

\section{Hechler forcing and $\lim^1 \mathbf{A}$} \label{lim_1_section}

We begin this section with a few useful observations about Hechler forcing and
finite-support iterations thereof. Recall that conditions in Hechler forcing are
pairs $p = (s_p, f_p)$, where $s_p \in {^{<\omega}}\omega$ and
$f_p \in {^\omega}\omega$. We will often call $s_p$ the \emph{stem} of $p$. If $p$ and $q$
are conditions then $q \leq p$ if and only if \begin{itemize}
\item $s_q \supseteq s_p$,
\item $f_q \geq f_p$, and
\item $s_q(i) > f_p(i)$ for all
$i \in \dom(s_q) \setminus \dom(s_p)$.\end{itemize}
Observe that any two conditions with the same stem are compatible.

If $\bb{P}$ is a finite-support iteration
of Hechler forcings, then the set of $p \in \bb{P}$ such that, for all
$\alpha \in \dom(p)$,
\begin{itemize}
  \item $p \restriction \alpha$ decides the value of $\dot{s}_{p(\alpha)}$, and
  \item $\dot{f}_{p(\alpha)}$ is a nice $\bb{P}_\alpha$-name for an element of
  ${^\omega} \omega$
\end{itemize}
is dense in $\bb{P}$. We will in general implicitly assume
that all of our conditions come from this dense set.

The following situation will frequently arise:

\begin{lemma} \label{compatibility_lemma}
  Suppose that
  \begin{itemize}
    \item $\bb{P}$ is a finite-support iteration of Hechler forcings,
    \item $\mathcal{A}$ is a finite collection of conditions in $\bb{P}$ such
    that $s_{p(\alpha)} = s_{q(\alpha)}$ for all $p,q \in \mathcal{A}$ and all $\alpha \in \dom(p) \cap
    \dom(q)$,
    \item $r \in \bb{P}$ and
    $r \leq p \restriction (\max(\dom(r)) + 1)$ for every $p \in \mathcal{A}$.
  \end{itemize}
  Then the set $\mathcal{A} \cup \{r\}$ has a lower bound in $\bb{P}$.
\end{lemma}

\begin{proof}
  We define such a lower bound $q$ in $\bb{P}$ as follows. First, let $\dom(q) =
  \dom(r) \cup \bigcup_{p \in \mathcal{A}} \dom(p)$. Next, let $q
  \restriction (\max(\dom(r)) + 1) = r$. Finally, for all $\alpha \in \dom(q)
  \setminus (\max(\dom(r)) + 1)$, let $s_{q(\alpha)} = s_{p(\alpha)}$ for
  any $p \in \mathcal{A}$ with $\alpha \in \dom(p)$ and let $\dot{f}_{q(\alpha)}$ be a nice
  $\bb{P}_\alpha$-name for the pointwise maximum of the set $\{\dot{f}_{p(\alpha)}
  \mid p \in \mathcal{A}, ~ \alpha \in \dom(p)\}$. It is now straightforward to verify that
  $q$ is a condition in $\bb{P}$ and as desired.
\end{proof}

We now turn to proving a lemma about uniformizing families of conditions in
finite support iterations of Hechler forcings. We first need the following
consequence of weak compactness.

\begin{lemma} \label{weakly_compact_lemma}
  Suppose that $\kappa$ is a weakly compact cardinal, $n < \omega$, and that
  $\left\langle X_{\vec{\alpha}} ~\middle|~ \vec{\alpha} \in [\kappa]^{\leq n}\right\rangle$ is a family of
  sets such that $|X_{\vec{\alpha}}| < \kappa$ for every $\vec{\alpha} \in
  [\kappa]^{\leq n}$. Suppose also that $c$ is a function with domain
  $[\kappa]^{n+1}$ such that, for all $\vec{\beta} \in [\kappa]^{n+1}$,
  we have $c(\vec{\beta}) \in \prod_{i \leq n} X_{\vec{\beta} \restriction i}$.
  Then there is an unbounded set $A \subseteq \kappa$ such that, for all
  $i \leq n$, all $\vec{\alpha} \in [A]^i$, and all $\vec{\gamma}, \vec{\delta}
  \in [A]^{n+1}$ end-extending $\vec{\alpha}$, we have $c(\vec{\gamma})(i) =
  c(\vec{\delta})(i)$.
\end{lemma}

\begin{proof}
  The proof is a relatively straightforward modification of the
  classical ramification argument used to prove that $\kappa$ satisfies
  $\kappa \rightarrow (\kappa)^n_\theta$ for all $n < \omega$ and
  $\theta < \kappa$. We provide some details for completeness.

  We proceed by induction on $n$. If $n = 0$, then the result follows simply
  from the fact that $\kappa$ is regular and $|X_\emptyset| < \kappa$.
  Thus, suppose $n = m + 1$ and we have established all instances of the
  result for $m$. By standard arguments (cf.\ \cite[Lemma 7.2]{kanamori}),
  using the strong inaccessibility of $\kappa$ and the fact that
  $|X_{\vec{\alpha}}| < \kappa$ for all $\vec{\alpha} \in [\kappa]^{\leq n}$,
  we can define a tree ordering $(\kappa, <_c)$ such that
  \begin{enumerate}
    \item for all $\alpha < \beta < \kappa$, if $\alpha <_c \beta$, then
    $\alpha < \beta$;
    \item for all $\alpha_0 <_c \alpha_1 <_c \ldots <_c \alpha_m <_c
    \gamma <_c \delta$, we have $c(\langle \alpha_0, \ldots, \alpha_m,
    \gamma \rangle) = c(\langle \alpha_0, \ldots, \alpha_m, \delta \rangle)$;
    \item $(\kappa, <_c)$ is a $\kappa$-tree, i.e., it has height $\kappa$
    and every level has cardinality less than $\kappa$.
  \end{enumerate}
  Since $\kappa$ is weakly compact and therefore has the tree property, we can
  fix an unbounded $A_0 \subseteq \kappa$ such that $A_0$ is linearly ordered by
  $<_c$. Now define a function $d$ with domain $[A_0]^n$ as follows.
  For every $\vec{\alpha} \in [A_0]^n$, let
  $d(\vec{\alpha}) = c(\vec{\alpha} ^\frown \langle \gamma \rangle) \restriction
  n$ for some $\gamma \in A_0 \setminus (\vec{\alpha}(m) + 1)$. Note that,
  by property (2) of $(\kappa, <_c)$ and the fact that $A_0$ is linearly ordered
  by $c$, the value of $d(\vec{\alpha})$ is independent of our choice of
  $\gamma$.

  Apply the inductive hypothesis to find an unbounded $A \subseteq A_0$ such that,
  for all $i \leq m$, all $\vec{\eta} \in [A]^i$, and all $\vec{\alpha}$,
  $\vec{\beta} \in [A]^n$ end-extending $\vec{\eta}$, we have $d(\vec{\alpha})(i) =
  d(\vec{\beta})(i)$. We claim that $A$ is as desired. To this end, fix
  $i \leq n$, $\vec{\alpha} \in [A]^i$, and $\vec{\gamma}, \vec{\delta}
  \in [A]^{n+1}$ end-extending $\vec{\alpha}$. If $i = n$, then the fact that $A$
  is linearly ordered by $<_c$ implies that
  $c(\vec{\gamma}) = c(\vec{\delta})$, so, in particular, $c(\vec{\gamma})(i) =
  c(\vec{\delta})(i)$. If $i < n$, then, by our definition of $d$ and our choice of $A$, we have
  \[
    c(\vec{\gamma})(i) = d(\vec{\gamma} \restriction n)(i) =
    d(\vec{\delta} \restriction n)(i) = c(\vec{\delta})(i),
  \]
  as required.
\end{proof}

\begin{lemma} \label{uniformizing_lemma}
  Let $n$ be a positive integer and let $\kappa$ be a weakly compact cardinal.
  Let $\left\langle p_{\vec{\alpha}}~\middle|~\vec{\alpha} \in [\kappa]^n\right\rangle$ be a family of
  conditions in $\bb{P}$, a length-$\kappa$ finite-support
  iteration of Hechler forcings. Let $u_{\vec{\alpha}} = \dom(p_{\vec{\alpha}})$.
  Then there is an unbounded set $A \subseteq \kappa$, a family
  $\langle u_{\vec{\alpha}} \mid \vec{\alpha} \in [A]^{<n}
  \rangle$, a natural number $\ell$, and a set of stems $\langle s_i \mid i < \ell \rangle$
  such that
  \begin{enumerate}
    \item $|u_{\vec{\alpha}}| =
    \ell$ for all $\vec{\alpha} \in [A]^n$, and if $\eta$ is the $i^{\mathrm{th}}$ element of
    $u_{\vec{\alpha}}$ then $s_{p_{\vec{\alpha}}(\eta)} = s_i$,
    \item $A$ and $\langle u_{\vec{\alpha}} \mid \vec{\alpha} \in [A]^{\leq n} \rangle$
    satisfy the conclusions of Lemma \ref{strong_delta_system_lemma},
    \item $p_{\vec{\beta}} \restriction u_{\vec{\alpha}}
    = p_{\vec{\gamma}} \restriction u_{\vec{\alpha}}$ for all $\vec{\alpha} \in [A]^{<n}$ and $\vec{\beta}, \vec{\gamma} \in
    [A]^n$ such that $\vec{\alpha} \sqsubseteq \vec{\beta}$ and $\vec{\alpha}
    \sqsubseteq \vec{\gamma}$.
  \end{enumerate}
\end{lemma}

\begin{proof}
  First, apply the weak compactness of $\kappa$ to the function on
  $[\kappa]^n$ defined by
  \[
    \vec{\alpha} \mapsto \langle s_{p_{\vec{\alpha}}(\eta)} \mid \eta \in
    \dom(p_{\vec{\alpha}}) \rangle
  \]
  to find an unbounded set $A_0 \subseteq \kappa$, a natural number
  $\ell < \omega$, and a sequence
  $\langle s_i \mid i < \ell \rangle$ of stems such that, for all $\vec{\alpha}
  \in A_0$, we have $|\dom(p_{\vec{\alpha}})| = \ell$ and
  \[
    \langle s_{p_{\vec{\alpha}}(\eta)} \mid \eta \in \dom(p_{\vec{\alpha}}) \rangle
    = \langle s_i \mid i < \ell \rangle.
  \]

  Next, apply Lemma \ref{strong_delta_system_lemma} to $\langle u_{\vec{\alpha}}
  \mid \vec{\alpha} \in [A_0]^n \rangle$ to find an unbounded set
  $A_1 \subseteq A_0$ and sets $\langle u_{\vec{\alpha}} \mid
  \vec{\alpha} \in [A_1]^{<n} \rangle$ as in the conclusion of that lemma.

  For each $\vec{\alpha} \in [A_1]^{<n}$ let $X_{\vec{\alpha}} =
  \{p \in \bb{P} \mid \dom(p) = u_{\vec{\alpha}}\}$ and consider the function $c$
  on $[A_1]^n$ defined as follows. For every $\vec{\beta} \in [A_1]^n$,
  define $c(\vec{\beta}) \in \prod_{i < n} X_{\vec{\beta} \restriction i}$
  by letting $c(\vec{\beta})(i) = p_{\vec{\beta}} \restriction u_{\vec{\beta}
  \restriction i}$ for all $i < n$.

  Apply Lemma~\ref{weakly_compact_lemma}
  to find an unbounded $A \subseteq A_1$ such that, for all
  $i < n$, all $\vec{\alpha} \in [A]^i$, and all $\vec{\beta}, \vec{\gamma}
  \in [A]^n$ end-extending $\vec{\alpha}$, we have
  $c(\vec{\beta})(i) = c(\vec{\gamma})(i)$, i.e., $p_{\vec{\beta}} \restriction
  u_{\vec{\alpha}} = p_{\vec{\gamma}} \restriction u_{\vec{\alpha}}$. Then
  $A$, $\langle u_{\vec{\alpha}} \mid \vec{\alpha} \in [A]^{\leq n} \rangle$,
  $\ell$, and $\langle s_i \mid i < \ell \rangle$ are as desired.
\end{proof}

\begin{remark}
  Suppose that $A$, $\langle u_{\vec{\alpha}}
  \mid \vec{\alpha} \in [A]^{\leq n} \rangle$, $\ell$, and $\langle s_i \mid i < \ell
  \rangle$ are as in the conclusion of Lemma \ref{uniformizing_lemma}.
  Given $\vec{\alpha} \in [A]^{<n}$, define $p_{\vec{\alpha}} \in \bb{P}$ by
  choosing any $\vec{\beta} \in [A]^n$ with $\vec{\alpha} \sqsubseteq
  \vec{\beta}$ and letting $p_{\vec{\alpha}} = p_{\vec{\beta}} \restriction
  u_{\vec{\alpha}}$ (observe that this is independent of our choice of $\vec{\beta}$).
  The family $\left\langle p_{\vec{\alpha}}~\middle|~\vec{\alpha}\in [A]^{\leq n}\right\rangle$ then has the property that if $m\leq n$ and $\vec{\alpha}, \vec{\beta} \in [A]^m$ are aligned, then $p_{\vec{\alpha}}$
  and $p_{\vec{\beta}}$ are compatible in $\bb{P}$.
\end{remark}

We turn now to the first of our results on the vanishing of $\lim^n\mathbf{A}$. In
this case, the assumption of the existence of a weakly compact cardinal is
certainly more than we need.
For example, it is shown in \cite{kamo} that $\lim^1 \mathbf{A} = 0$ in any forcing
extension by $\mathrm{Add}(\omega, \omega_2)$, the poset to add $\aleph_2$-many
Cohen reals. A similar argument shows that $\lim^1 \mathbf{A} = 0$ in any forcing
extension by a length-$\omega_2$ finite support iteration of Hechler forcings, at
least as long as the ground model satisfies $\mathsf{CH}$.
This is to say that our main purpose in recording this theorem and proof is to foreground the essential logic of the more complicated arguments that will follow.

\begin{theorem} \label{hechler_theorem_one}
  Suppose that $\kappa$ is a weakly compact cardinal, and let $\bb{P}$ be a
  finite-support iteration of Hechler forcings of length $\kappa$. Then
  $V^{\bb{P}}\vDash\textnormal{``}\lim^1 \mathbf{A} = 0.\textnormal{''}$
\end{theorem}

\begin{proof}
  For $\alpha < \kappa$, let $\dot{g}_\alpha$ be a $\bb{P}$-name for the
  $\alpha^{\mathrm{th}}$ Hechler real added by $\bb{P}$. In $V^{\bb{P}}$,
  $\langle \dot{g}_\alpha \mid \alpha < \kappa \rangle$ is a $<^*$-increasing,
  cofinal sequence in ${^\omega}\omega$, hence it suffices to show that, in $V^{\bb{P}}$,
  every coherent family indexed by $\langle \dot{g}_\alpha \mid \alpha < \kappa \rangle$
  is trivial. To this end, fix a condition $p_0 \in \bb{P}$ and a sequence of
  $\bb{P}$-names $\dot{\Phi} = \left\langle \dot{\varphi}_\alpha : I(\dot{g}_\alpha) \rightarrow
  \bb{Z} ~ \middle| ~ \alpha < \kappa \right\rangle$ forced by $p_0$ to be a coherent family.
  We will find a $q \leq p_0$ forcing $\dot{\Phi}$ to be trivial.

  Let $A_0 = \kappa \setminus (\max(\dom(p_0)) + 1)$. Observe that for any
  $\alpha < \beta$ in $A_0$ there exists a $p \leq p_0$ such that
  $p \Vdash_{\bb{P}}``\dot{g_\alpha} \leq \dot{g_\beta}"$. Namely, let $p = p_0 \cup \{\langle \beta, (\emptyset, \dot{g}_\alpha) \rangle\}$.
  Hence for each $\alpha < \beta$ in $A_0$, since $p_0$ forces that
  $\dot{\Phi}$ is coherent, we may extend $p_0$ to a $q_{\alpha, \beta}
  \leq p_0$ such that for some $k_{\alpha, \beta} <\omega$,
    \[
    q_{\alpha, \beta}
    \Vdash_{\bb{P}}``\dot{g_\alpha} \leq \dot{g_\beta},\text{ and }\dot{\varphi}_\alpha(k,m)
    = \dot{\varphi}_\beta(k,m)\text{ for all } k \geq k_{\alpha, \beta}
    \text{ and } m \leq \dot{g}_\alpha(k)."
    \]
  Thin out $A_0$, if necessary, to an unbounded $A_1$ such that
  $k_{\alpha, \beta}$ equals some fixed $k^*$ for all $\langle \alpha, \beta \rangle \in [A_1]^2$.
  Now apply Lemma \ref{uniformizing_lemma} to $\big\langle q_{\alpha, \beta}
  \mid \langle \alpha, \beta \rangle \in [A_1]^2 \big\rangle$ to find an unbounded
  $A \subseteq A_1$, $u_\emptyset$, $\langle u_{\langle \alpha \rangle} \mid
  \alpha \in A \rangle$, $\ell$, and $\langle s_i \mid i < \ell \rangle$ as
  in the statement of the lemma.

  Define conditions $q_\emptyset$ and
  $\langle q_\alpha \mid \alpha \in A \rangle$ as follows. First, let
  $\alpha < \beta$ be elements of $A$ and let
  $q_\emptyset = q_{\alpha, \beta} \restriction u_\emptyset$. By Lemma \ref{uniformizing_lemma} item (3), this definition is independent of the choice of $\alpha$ and $\beta$.
  Next, for a fixed $\alpha \in A$, choose an arbitrary $\beta \in A$
  with $\beta > \alpha$ and let $q_\alpha = q_{\alpha, \beta} \restriction
  u_{\langle \alpha \rangle}$. Here again the choice of $\beta$ is immaterial.

  Let $q = q_\emptyset$ and notice that, since $q_{\alpha, \beta} \leq p_0$ for all
  $\alpha < \beta$ in $A_0$, we also have $q \leq p_0$.
  We claim that $q$ forces $\dot{\Phi}$ to be trivial. Let $\dot{B}$
  be a $\bb{P}$-name for the set of $\alpha \in A$ such that $q_\alpha \in \dot{G}$,
  where $\dot{G}$ is the canonical $\bb{P}$-name for the $\bb{P}$-generic filter.

  \begin{claim} \label{unbounded_b_claim}
    $q \Vdash_{\bb{P}}``\dot{B} \text{ is unbounded in } \kappa"$.
  \end{claim}

  \begin{proof}
    Let $r \leq q$ and $\eta < \kappa$ be arbitrary. It suffices to find an
    $\alpha \in (A \setminus \eta)$ such that $q_\alpha$ and $r$ are compatible.
    By construction, the sequence $\left\langle u_{\langle \alpha \rangle}
    \setminus u_\emptyset ~ \middle| ~ \alpha \in A \right\rangle$ consists of pairwise
    disjoint sets, so $(u_{\langle \alpha \rangle} \setminus u_\emptyset) \,\cap\, \dom(r) = \emptyset$ for some $\alpha \in (A \setminus \eta)$.
    This implies that $q_\alpha$ is compatible with $r$, since $q_\alpha \restriction u_\emptyset
    = q_\emptyset$ and $r \leq q_\emptyset$.
  \end{proof}

  \begin{claim} \label{linking_claim_1}
    $q$ forces that, for all $\alpha < \beta$ in $\dot{B}$, there is a
    $\gamma \in (A \setminus \beta + 1)$ such that $q_{\alpha, \gamma}$
    and $q_{\beta, \gamma}$ are in $\dot{G}$.
  \end{claim}

  \begin{proof}
    Fix $r \leq q$ and $\alpha < \beta$ such that $r$ forces both $\alpha$
    and $\beta$ to be in $\dot{B}$. By the definition of $\dot{B}$, we may
    assume that $r$ extends both $q_\alpha$ and $q_\beta$. It then suffices
    to find a $\gamma \in (A \setminus \beta + 1)$ such that $r, q_{\alpha, \gamma}$,
    and $q_{\beta, \gamma}$ all have a common extension $s$. By construction, the
    families $\left\langle u_{\langle \alpha, \gamma \rangle} \setminus
    u_{\langle \alpha \rangle} ~ \middle| ~ \gamma \in (A \setminus \beta + 1) \right\rangle$
    and $\left\langle u_{\langle \beta, \gamma \rangle} \setminus u_{\langle \beta \rangle}
    ~ \middle| ~ \gamma \in (A \setminus \beta + 1) \right\rangle$ each consist of pairwise
    disjoint sets. We can therefore find a $\gamma \in (A \setminus \beta + 1)$
    such that $(u_{\langle \alpha, \gamma \rangle} \setminus
    u_{\langle \alpha \rangle})\, \cap\, \dom(r) = \emptyset$ and
    $(u_{\langle \beta, \gamma \rangle} \setminus
    u_{\langle \beta \rangle}) \,\cap\, \dom(r) = \emptyset$.
    By item (2b) of Lemma \ref{strong_delta_system_lemma}, $u_{\langle \alpha, \gamma \rangle}$
    and $u_{\langle \beta, \gamma \rangle}$ are aligned, hence by item (1) of Lemma \ref{uniformizing_lemma}, the stems in $q_{\alpha,\gamma}$ and $q_{\beta,\gamma}$ match whenever their domains intersect. Now apply Lemma \ref{compatibility_lemma} to $r$ and
    $\left\{q_{\alpha, \gamma}, q_{\beta, \gamma }\right\}$
    to find a single condition $s$ extending all three conditions.
    This is the condition that we had sought.
  \end{proof}

  \begin{claim} \label{claim_48}
    $q$ forces that, for all $\alpha < \beta$ in $\dot{B}$,
    $$\dot{\varphi}_\alpha(k,m) = \dot{\varphi}_\beta(k,m)$$
    for every $k \in \omega \setminus k^*$ and $m \leq \min(\{\dot{g}_\alpha(k), \dot{g}_\beta(k)\})$.
  \end{claim}

  \begin{proof}
    Let $G$ be $\bb{P}$-generic with $q \in G$. Work in $V[G]$. Fix
    $\alpha < \beta$ in $B$. By Claim
    \ref{linking_claim_1}, we can find $\gamma \in (A \setminus \beta + 1)$ such
    that $q_{\alpha, \gamma}$ and $q_{\beta, \gamma}$ are in $G$. This implies that
    \begin{itemize}
      \item $g_\alpha \leq g_\gamma$,
      \item $g_\beta \leq g_\gamma$, and
      \item $\varphi_\alpha(k,m) = \varphi_\gamma(k,m) = \varphi_\beta(k,m)$ for all $k \geq k^*$ and $m \leq \min\{g_\alpha(k),
      g_\beta(k)\}$.
    \end{itemize}
    The claim follows.
  \end{proof}
  Hence if $G$ is $\bb{P}$-generic with $q \in G$, then we may define a function
  $\psi : \omega \times \omega \rightarrow \omega$ trivializing $\Phi$ as follows.
  For all $(k,m) \in \omega \times \omega$, if there is an $\alpha \in B$ with
  $m \leq g_\alpha(k)$, then let $\psi(k,m) =
  \varphi_\alpha(k,m)$ for any such $\alpha$. Note that, by Claim~\ref{claim_48},
  if $k \geq k^*$, then the value of $\psi(k,m)$ is independent of our choice of
  $\alpha$. If there is no such $\alpha \in B$, let $\psi(k,m)=0$ (or any other integer). It is then straightforward to verify that $\psi$ does in fact
  trivialize $\Phi$, thus completing the proof of the theorem.
\end{proof}

\section{Hechler forcing and $\lim^2 \mathbf{A}$}\label{lim_2 section}

In this section, we prove the $n = 2$ case of our Main Theorem, which contains
most of the key ideas of the general proof but is significantly less complex.
We begin with some preliminary observations.

Suppose that $\Phi = \left\langle \varphi_{fg} : I(f \wedge g) \rightarrow \bb{Z}
~ \middle| ~ f,g \in {^\omega}\omega \right\rangle$ is a $2$-coherent family of functions.
Recall that $\Phi$ determines a function $\partial^1\Phi$ taking each $(f,g,h)\in ({^\omega}\omega)^3$ to the function
\[
  \varphi_{gh} - \varphi_{fh} + \varphi_{fg} : I(f \wedge g \wedge h)
  \rightarrow \bb{Z}
\]
For simplicity, in what follows we write $e$ for the difference-function $\partial^n\Phi$ whenever $\Phi$ and $n$ are clear. Observe that when $\Phi$ is 2-coherent, every $e(f,g,h)$ is finitely supported; in this case, we write $\mathtt{e}(f,g,h)$ for the restriction of $e(f,g,h)$ to its support.

Recall also that, by the lemmas in Section \ref{set_background}, $\Phi$ is trivial
if and only if there exists a
$\leq^*$-cofinal family $F \subseteq {^\omega}\omega$ and an alternating family
\[
\left\langle \psi_{fg}:I(f \wedge g) \rightarrow \bb{Z}
~ \middle| ~ f,g \in F \right\rangle
\]
of finitely supported functions such that
for all $f,g,h \in F$
\[
  e(f,g,h) = (\psi_{gh} - \psi_{fh} + \psi_{fg})\restriction I(f\wedge g\wedge h).
\]
When the functions in question are from an enumerated sequence $\langle f_\alpha \mid \alpha < \kappa \rangle$ in ${^\omega}\omega$, we will often write $e(\alpha, \beta, \gamma)$ in place of $e(f_\alpha, f_\beta, f_\gamma)$, and
$\varphi_{\alpha \beta}$ in place of $\varphi_{f_\alpha f_\beta}$,
and so on. In such cases, due to the fact that all of our families of functions
will be alternating, it will suffice to deal only with functions
$\varphi_{\alpha \beta}$ for $\alpha < \beta$ and $e(\alpha, \beta, \gamma)$
for $\alpha < \beta < \gamma$, since these determine the rest of the functions
in the relevant families. A similar statement will hold in the general case
in Section \ref{lim_n section}.

\begin{theorem} \label{hechler_theorem_two}
  Suppose that $\kappa$ is a weakly compact cardinal, and let $\bb{P}$ be a
  finite-support iteration of Hechler forcing of length $\kappa$. Then $V^{\bb{P}}\vDash\textnormal{``}\lim^2 \mathbf{A} = 0.\textnormal{''}$
\end{theorem}

\begin{proof}
  As in the proof of Theorem \ref{hechler_theorem_one}, for all $\alpha < \kappa$, let
  $\dot{g}_\alpha$ be a nice $\bb{P}$-name for the Hechler real added at the $\alpha^{\textnormal{th}}$ stage of $\mathbb{P}$. Much as before, it will suffice to consider 2-coherent families in $V^{\bb{P}}$ indexed
  by pairs of functions from $\langle \dot{g}_\alpha \mid \alpha < \kappa \rangle$.
  Therefore fix a condition $p_0 \in \bb{P}$ and a sequence of $\bb{P}$-names
  $\dot{\Phi} = \left\langle \dot{\varphi}_{\alpha \beta}:I(\dot{g}_\alpha \wedge
  \dot{g}_\beta) \rightarrow \bb{Z} ~ \middle| ~ \alpha < \beta < \kappa \right\rangle$ forced
  by $p_0$ to be a 2-coherent family. We will find a $q \leq p_0$ which forces
  $\dot{\Phi}$ to be trivial.

  Let $A_0 = \kappa \setminus (\max(\dom(p_0)) + 1)$. Again, much as before, for all $\alpha < \beta < \gamma$ in $A_0$,
  there exists a condition $q_{\alpha, \beta, \gamma} \leq p_0$ such that
  $q_{\alpha, \beta, \gamma} \Vdash_{\bb{P}} ``\dot{g}_\alpha \leq \dot{g}_\beta
  \leq \dot{g}_\gamma"$ and $q_{\alpha, \beta, \gamma}$ decides the value of
  $\dot{\mathtt{e}}(\alpha, \beta, \gamma)$ to be  equal to some $\epsilon_{\alpha, \beta, \gamma} \in V$.
  This condition exists because $\dot{e}(\alpha, \beta, \gamma)$ is forced by $p_0$ to
  be a finitely supported function. 
  Thin out $A_0$ to an unbounded $A_1$ such that
  $\epsilon_{\alpha, \beta, \gamma}$ equals some fixed $\epsilon$ for all
  $\langle \alpha, \beta, \gamma \rangle \in [A_1]^3$.

  Now apply Lemma \ref{uniformizing_lemma} to $\left\langle q_{\alpha, \beta, \gamma}
  ~ \middle| ~ \langle \alpha, \beta, \gamma \rangle \in [A_1]^3 \right\rangle$ to find
  an unbounded $A \subseteq A_1$
  together with sets $u_\emptyset$, $\left\langle u_{\langle \alpha \rangle} ~ \middle| ~
  \alpha \in A \right\rangle$, and $\left\langle u_{\langle \alpha, \beta \rangle}  ~ \middle| ~
  \alpha < \beta \in A \right\rangle$, a natural number $\ell$, and stems
  $\langle s_i \mid i < \ell \rangle$ as in the statement of the lemma. Define
  conditions $q_\emptyset$, $\langle q_\alpha \mid \alpha \in A \rangle$, and
  $\langle q_{\alpha, \beta} \mid \alpha < \beta \in A \rangle$ as follows.
  First, for any $\alpha < \beta < \gamma$ in $A$ let
  $q_\emptyset = q_{\alpha, \beta, \gamma} \restriction u_\emptyset$. Next,
  for each fixed $\alpha \in A$, choose $\beta < \gamma$ in $A$ with
  $\beta > \alpha$ and let $q_\alpha = q_{\alpha, \beta, \gamma} \restriction
  u_{\langle \alpha \rangle}$. Finally, for fixed $\alpha < \beta$ in $A$,
  choose a  $\gamma \in A$ with $\gamma > \beta$ and let
  $q_{\alpha, \beta} = q_{\alpha, \beta, \gamma} \restriction
  u_{\langle \alpha, \beta \rangle}$. By item (3) of Lemma \ref{uniformizing_lemma}, these definitions
  are independent of all of our choices.

  Let $q = q_\emptyset$, and note that $q \leq p_0$. We claim that $q$ forces $\dot{\Phi}$ to be trivial.
  This we argue by first partitioning $A$ into two disjoint and unbounded subsets, $\Gamma_1$ and
  $\Gamma_2$. Let $\dot{B}$ be a $\bb{P}$-name for the set of
  $\alpha \in \Gamma_1$ such that $q_\alpha \in \dot{G}$, where $\dot{G}$ is the canonical name for the $\mathbb{P}$-generic filter. The proofs of the
  next two claims follow those of Claims \ref{unbounded_b_claim}
  and \ref{linking_claim_1} almost verbatim, so we omit them.

  \begin{claim} \label{unbounded_b_claim_2}
    $q \Vdash_{\bb{P}} ``\dot{B} \text{ is unbounded in } \kappa."$
  \end{claim}

  \begin{claim} \label{linking_claim_2}
    $q$ forces that, for all $\alpha < \beta$ in $\dot{B}$, there is a
    $\gamma \in (\Gamma_2 \setminus \beta + 1)$ such that
    $q_{\alpha, \gamma}$ and $q_{\beta, \gamma}$ are in $\dot{G}$.
  \end{claim}

  \begin{claim} \label{amalg_claim_2}
    $q$ forces the following to hold in $V^{\bb{P}}$: Suppose that
    $\tau$ is in $[\dot{B}]^3$ and $\langle\alpha_\sigma \mid \sigma \in [\tau]^2 \rangle$
    is such that, for all $\sigma \in [\tau]^2$ and all $\eta \in \sigma$,
    \begin{itemize}
      \item $\alpha_\sigma \in \Gamma_2 \setminus (\eta + 1)$, and
      \item $q_{\eta, \alpha_{\sigma}}\in\dot{G}$.
    \end{itemize}
    Then there is an $\alpha_{\tau}$ in $A \setminus \left(\max\{\alpha_\sigma \mid
    \sigma \in [\tau]^2\} + 1\right)$
    such that, for all $\sigma \in [\tau]^2$ and $\eta \in \sigma$,
    we have $q_{\eta, \alpha_\sigma, \alpha_\tau} \in \dot{G}$.
  \end{claim}

  \begin{proof}
    Fix $\tau \in [\Gamma_1]^3$, $\{\alpha_\sigma \mid \sigma \in [\tau]^2\}$,
    and $r \leq q$ such that $r$ forces $\tau$ and $\{\alpha_\sigma \mid
    \sigma \in [\tau]^2\}$ to be as in the premise of the claim.
    Much as before, we may assume that $r$ extends $q_{\eta, \alpha_\sigma}$
    for all $\sigma \in [\tau]^2$ and $\eta \in \sigma$. By construction,
    for all such $\sigma$ and $\eta$, the sequence $\left\langle u_{\langle \eta,
    \alpha_\sigma, \gamma \rangle} \setminus u_{\langle \eta, \alpha_\sigma
    \rangle} ~ \middle| ~ \gamma \in (A \setminus \alpha_\sigma
    + 1) \right\rangle$ consists of pairwise disjoint sets. Hence there exists
    an $\alpha_\tau \in A \setminus \left(\max\{\alpha_\sigma \mid \sigma \in
    [\tau]^2\} + 1\right)$ such that
    \[
      \left(u_{\langle \eta, \alpha_\sigma, \alpha_\tau \rangle} \setminus
      u_{\langle \eta, \sigma \rangle}\right) \cap \dom(r) = \emptyset
    \]
    for all $\sigma \in [\tau]^2$ and $\eta \in \sigma$. Observe that,
    for any $\sigma, \rho \in [\tau]^2$, $\eta \in \sigma$, and $\xi \in \rho$,
    the sets $\{\eta, \alpha_\sigma\}$ and $\{\xi, \alpha_\rho\}$ are aligned
    (it was to ensure this alignment that we partitioned $A$ into the
    disjoint sets $\Gamma_1$ and $\Gamma_2$). Therefore, by item (2b) of Lemma
    \ref{strong_delta_system_lemma}, the sets $u_{\langle \eta, \alpha_\sigma,
    \alpha_\tau \rangle}$ and $u_{\langle \xi, \alpha_\rho, \alpha_\tau
    \rangle}$ are aligned, hence by item (1) of Lemma \ref{uniformizing_lemma}
    the conditions $q_{\eta, \alpha_\sigma, \alpha_\tau}$ and
    $q_{\xi, \alpha_\rho, \alpha_\tau}$ are compatible, with identical stems
    wherever their domains intersect. Now apply Lemma \ref{compatibility_lemma}
    to $\left\langle q_{\eta, \alpha_\sigma, \alpha_\tau} ~ \middle| ~ \sigma \in [\tau]^2, ~
    \eta \in \sigma \right\rangle$ and $r$ to find a single condition $s \in \bb{P}$ simultaneously
    extending all of these conditions. This $s$ forces $\alpha_\tau$ to be as
    desired.
  \end{proof}

  We now turn more directly to the argument that $q$ forces $\dot{\Phi}$ to be trivial.
  Let $G$ be a $\bb{P}$-generic filter with $q \in G$; work in $V[G]$, and let
  $B$ and $\Phi$ be the realizations of $\dot{B}$ and $\dot{\Phi}$. Since $B$ is cofinal in $\kappa$, the family $\{g_\eta\mid\eta\in B\}$ is $\leq^*$-cofinal in $^\omega\omega$. Therefore it will suffice to find a family
  $\Psi=\left\langle \psi_{\eta \xi}:I(g_\eta \wedge g_\xi)
  \rightarrow \bb{Z} ~ \middle| ~ \eta < \xi \in B \right\rangle$ of finitely
  supported functions such that,
  for all $\eta_0 < \eta_1 < \eta_2$ in $B$,
  \[
  e(\eta_0, \eta_1, \eta_2) = (\psi_{\eta_1 \eta_2}
  - \psi_{\eta_0 \eta_2} + \psi_{\eta_1 \eta_2})\restriction I(g_{\eta_0} \wedge
  g_{\eta_1}\wedge g_{\eta_2}).
  \]

  Using Claim \ref{linking_claim_2}, choose ordinals $\langle \alpha_\sigma
  \mid \sigma \in [B]^2 \rangle$ such that $\alpha_\sigma \in (\Gamma_2 \setminus
  \eta + 1)$ and $q_{\eta, \alpha_\sigma} \in G$ for all $\sigma \in [B]^2$
  and all $\eta \in \sigma$. Then, using
  Claim \ref{amalg_claim_2}, choose ordinals $\langle\alpha_\tau \mid
  \tau \in [B]^3 \rangle$ such that $\alpha_\tau \in (A \setminus \alpha_\sigma + 1)$ and
  $q_{\eta, \alpha_\sigma, \alpha_\tau} \in G$ for all $\tau \in [B]^3$, all
  $\sigma \in [\tau]^2$, and all $\eta \in \sigma$. Together these choices ensure that
  $g_\eta \leq g_{\alpha_\sigma} \leq q_{\alpha_\tau}$
  and $\mathtt{e}(\eta, \alpha_\sigma, \alpha_\tau) = \epsilon$ for all $\tau \in [B]^3$ and $\sigma
  \in [\tau]^2$ and $\eta \in \sigma$.

  Now for arbitrary elements $\eta < \xi$ in $B$ let $\psi_{\eta \xi} = e(\eta, \xi, \alpha_{\{\eta,\xi\}})$.
  Since $g_\eta$ and  $g_\xi$ are both less than or equal to $g_{\alpha_{\{\eta,\xi\}}}$, the domain
  of $\psi_{\eta \xi}$ is in fact $I(g_\eta \wedge g_\xi)$, as desired.
  We claim that the family $\Psi$ so defined witnesses the triviality of $\Phi$.

  To see this, let $\eta_0 < \eta_1 < \eta_2$ be arbitrary elements of $B$.
  To simplify notation, let $\alpha_i = \eta_i$ and $\alpha_{ij} =
  \alpha_{\{\eta_i, \eta_j\}}$ for $i < j < 3$, and let $\alpha_{012} =
  \alpha_{\{\eta_0, \eta_1, \eta_2\}}$. We then have the following series of equations, each of which follows immediately from an expansion of expressions followed by a cancellation of like terms:
  \begin{align*}\tag{$\mathbf{a}_{012}$}
    e(\alpha_1, \alpha_2, \alpha_{012})  - e(\alpha_0, \alpha_2, \alpha_{012})
    + e(\alpha_0, \alpha_1, \alpha_{012})  - e(\alpha_0,\alpha_1,\alpha_2) & = 0, &
  \\ \tag{$\mathbf{a}_{12}$}
     e(\alpha_2, \alpha_{12}, \alpha_{012})  - e(\alpha_1, \alpha_{12}, \alpha_{012})
     + e(\alpha_1, \alpha_2, \alpha_{012})  - e(\alpha_1,\alpha_2,\alpha_{12}) & = 0, &
 \\ \tag{$\mathbf{a}_{02}$}
    e(\alpha_2, \alpha_{02}, \alpha_{012}) - e(\alpha_0, \alpha_{02}, \alpha_{012})
    + e(\alpha_0, \alpha_2, \alpha_{012}) - e(\alpha_0,\alpha_2,\alpha_{02}) & = 0, &
   \\ \tag{$\mathbf{a}_{01}$}
    e(\alpha_1, \alpha_{01}, \alpha_{012}) - e(\alpha_0, \alpha_{01}, \alpha_{012})
    + e(\alpha_0, \alpha_1, \alpha_{012}) - e(\alpha_0,\alpha_1,\alpha_{01}) & = 0.
& \end{align*}
  The logic of the labeling is as follows: the operative indices in
  $\mathbf{a}_{012}$ are $0,1,2,$ and $012$. The operative indices in
  $\mathbf{a}_{12}$ are $1,2,12$, and $012$; similarly for $\mathbf{a}_{02}$
  and $\mathbf{a}_{01}$. Here we have preferred readable equations to perfectly
  rigorous ones and have therefore omitted restriction-notations; the essential
  observation about the domains of the above functions is simply that
  $I(g_{\alpha_0} \wedge g_{\alpha_1}\wedge g_{\alpha_2})$ is a subdomain of
  every one. This is because $g_\eta \leq g_{\alpha_\sigma} \leq g_{\alpha_\tau}$ for every
  $\tau \in [B]^3$, $\sigma \in [\tau]^2$, and $\eta \in \sigma$.

  By construction, the restriction of each of the first two terms of $\mathbf{a}_{12}$, $\mathbf{a}_{02}$,
  and $\mathbf{a}_{01}$ to their support is equal to $e$.
  In consequence, these terms all cancel in the sum
  $\mathbf{a}_{12}-\mathbf{a}_{02}+\mathbf{a}_{01}-\mathbf{a}_{012}$, which reduces in turn,
  via cancellation of the third terms of $\mathbf{a}_{12}$, $\mathbf{a}_{02}$,
  and $\mathbf{a}_{01}$ with the like terms in $\mathbf{a}_{012}$, to
  \begin{align}\label{certification1}
    e(\alpha_0,\alpha_1,\alpha_2)- e(\alpha_1,\alpha_2,\alpha_{12})+
    e(\alpha_0,\alpha_2,\alpha_{02})-
    e(\alpha_0,\alpha_1,\alpha_{01}) =0.
  \end{align}
  We may rewrite this equation as
  \[
  e(\alpha_0, \alpha_1, \alpha_2) = (\psi_{\alpha_1 \alpha_2}
  - \psi_{\alpha_0 \alpha_2} + \psi_{\alpha_1 \alpha_2})\restriction I(g_{\alpha_0} \wedge
  g_{\alpha_1}\wedge g_{\alpha_2}).
  \]
  This is the relation we had desired; this concludes our proof.
\end{proof}

\section{Hechler forcing and $\lim^n \mathbf{A}$}\label{lim_n section}

In this section, we prove our Main Theorem.
We prepare for the proof by first building up
some technical machinery, frequently referring back to the proof of Theorem
\ref{hechler_theorem_two} for motivation. For the following sequence of definitions, let $\kappa$ be a fixed regular uncountable cardinal. Ultimately, of course, $\kappa$ will denote the weakly compact cardinal of our main theorem, but that hypothesis is irrelevant to Definitions and Claims \ref{61} through \ref{SisC}.

\begin{definition}\label{61}
  For any nonempty $\tau$ in $[\kappa]^{<\omega}$, a \emph{subset-initial segment
  of $\tau$} is a sequence $\sigma_1 \subseteq \cdots \subseteq \sigma_m \subseteq \tau$
  such that
  \begin{itemize}
    \item $m \leq |\tau|$ and
    \item $|\sigma_i| = i$ for all $i$ with $1 \leq i \leq m$.
  \end{itemize}
  We write $\vec{\sigma} \vartriangleleft \tau$ to indicate that
  $\vec{\sigma}$ is a subset-initial segment of $\tau$. If $m = |\tau|$, then
  we call $\vec{\sigma} = \langle \sigma_1, \cdots, \sigma_m \rangle\vartriangleleft\tau$ a \emph{long string} or
  \emph{long string for $\tau$}.
\end{definition}

Suppose now that $\langle g_\alpha \mid \alpha
< \kappa \rangle$ is an injective sequence of elements of ${^\omega}\omega$.  Suppose that for each positive integer $n$ the family
\[
  \Phi_n = \left\langle \varphi_{\vec{\alpha}}:I(\vec{\alpha})\to\mathbb{Z} ~ \middle| ~
  \vec{\alpha} \in [\kappa]^n \right\rangle
\]
is $n$-coherent, where $I(\vec{\alpha}) = I\left(\bigwedge_{i < n}\, g_{\alpha_i}\right)$
for each such $n$ and $\vec{\alpha} \in [\kappa]^n$. Let
$\vec{\Phi}$ denote the family $\left\langle \varphi_{\vec{\alpha}} ~ \middle| ~
\vec{\alpha} \in [\kappa]^{<\omega}, ~ \vec{\alpha} \neq \emptyset \right\rangle$.
Suppose also that $B \subseteq \kappa$ is unbounded and, to each nonempty $\tau \in [B]^{<\omega}$ we have assigned
an ordinal $\alpha_\tau < \kappa$ in such a way that
\begin{itemize}
  \item if $\tau = \{\eta\}$, then $\alpha_\tau = \eta$ and
  \item if $\rho \subsetneq \tau$, then $\alpha_\rho < \alpha_\tau$.
\end{itemize}
Given a nonempty $\tau \in [B]^{<\omega}$ and a subset-initial segment
$\vec{\sigma} \vartriangleleft \tau$, write $\vec{\alpha}[\vec{\sigma}]$
to denote the sequence $\langle \alpha_{\sigma_i} \mid 1 \leq i \leq
|\vec{\sigma}| \rangle$; note that this sequence is increasing by assumption.

For $\vec{\alpha} \in [\kappa]^{<\omega}$ of length at least two, let
\[
  e^{\vec{\Phi}}(\vec{\alpha}) = \sum_{i < |\vec{\alpha}|} (-1)^i
  \varphi_{\vec{\alpha}^i}
\]
When the family $\vec{\Phi}$ is clear from context, it will
be omitted from the superscript above. Similarly, we will continue to notationally suppress the restriction of sums of functions to the intersection of their domains.
Since each family $\Phi_n$ is $n$-coherent, the function $e(\vec{\alpha})$ is finitely supported for
any $\vec{\alpha} \in [\kappa]^{<\omega}$ of length at least two. We write $\mathtt{e}(\vec{\alpha})$ for the restriction of $e(\vec{\alpha})$ to its support.

We now record a series of formal equalities. First, for all $\vec{\alpha}
\in [\kappa]^{<\omega}$ of length at least three, let
\[
  \mathsf{d} e(\vec{\alpha}) = \sum_{i<|\vec{\alpha}|}(-1)^i e(\vec{\alpha}^i).
\]
Observe that when the right-hand side of the above equation is fully expanded,
its terms will cancel; $\mathsf{d} e$ is, after all, a composition of differentials.
Hence $\mathsf{d} e(\vec{\alpha})$ is well-defined and equals $0$ for all $\vec{\alpha}
\in [\kappa]^{<\omega}$ of length at least three.
If $L$ is some linear combination of the form
\[
  \sum_{i<\ell} a_i\, e(\vec{\alpha}_i)
\]
with all $\vec{\alpha}_i$ of length at least three then let
\[
  \mathsf{d} L=\sum_{i<\ell} a_i\, \mathsf{d} e(\vec{\alpha}_i).
\]
Similarly, if $j$ is less than $|\vec{\alpha}_i|$ for all $i$ then let
\[
  L^j=\sum_{i<\ell} a_i\, e(\vec{\alpha}^j_i).
\]
Finally, here dropping the assumption that each $\vec{\alpha}_i$ has length at
least three, for any $\vec{\beta}\in [\kappa]^{<\omega}$ with $\vec{\alpha}_i <
\vec{\beta}$ for all $i < \ell$, let
\[
  L*\vec{\beta}=\sum_{i<\ell} a_i\, e(\vec{\alpha}_i, \vec{\beta})
\]
 where $e(\vec{\alpha}_i, \vec{\beta}) = e(\vec{\alpha}_i
{^\frown} \vec{\beta})$. If $\vec{\beta} = \langle \beta \rangle$, then
we will abuse notation and write $L*\beta$ and $e(\vec{\alpha}_i, \beta)$
instead of $L*\langle \beta \rangle$ and $e(\vec{\alpha}_i, \langle \beta
\rangle)$.
Finally, if $j$ is less than $|\vec{\alpha}_i|$ for all $i$ then let
\[
  L^j=\sum_{i<\ell} a_i\, e(\vec{\alpha}^j_i)
\]

For integers $n\geq 2$ we now recursively define interrelated
\begin{itemize}
\item expressions $\mathcal{A}^{\vec{\Phi}}_n(\rho)$ parametrized by $\rho \in [B]^n$,
\item expressions $\mathcal{S}^{\vec{\Phi}}_n(\tau)$ and
$\mathcal{C}^{\vec{\Phi}}_n(\tau)$ parametrized by $\tau \in [B]^{n+1}$, and
\item statements $\mathfrak{u}^{\vec{\Phi}}_n(\tau)$ parametrized by $\tau \in [B]^{n+1}$.
\end{itemize}
Again when the family $\vec{\Phi}$ is clear from context it is omitted from superscripts. In fact the above expressions will depend also on the
collection $\langle \alpha_\tau \mid \tau \in [B]^{<\omega}\rangle$ fixed earlier, but this dependence is always plain enough that we ignore it, notationally, entirely.

To begin, let
\[
  \mathcal{A}_2(\rho) = e(\rho, \alpha_\rho)
\]
for each $\rho \in [B]^2$. Recall that we interpret elements of $[\kappa]^{<\omega}$
as finite increasing sequences, so, for example, if $\rho = \langle \beta_0, \beta_1 \rangle$,
then $e(\rho, \alpha_\rho)$ denotes $e(\langle \beta_0, \beta_1, \alpha_\rho \rangle)$.
Next, given $n$ with $2 \leq n < \omega$ and $\tau \in [B]^{n+1}$, if
$\mathcal{A}_n(\tau^i)$ has been defined for all $i \leq n$, let
\begin{align*}
  \mathcal{S}_n(\tau) &= \mathsf{d} e(\tau, \alpha_\tau) -
  \sum_{i < n + 1} (-1)^i \mathsf{d}[\mathcal{A}_n(\tau^i) * \alpha_\tau], \\
  \mathcal{C}_n(\tau) &= e(\tau) - \sum_{i < n + 1}
  (-1)^i \mathcal{A}_n(\tau^i),
\end{align*}
and let $\mathfrak{u}_n(\tau)$ denote the conjunction of the following two
statements:
\begin{itemize}
  \item There exists an $\epsilon$ such that $\mathtt{e}(\vec{\alpha}
  [\vec{\sigma}]) = \epsilon$ for every long string $\vec{\sigma}
  \vartriangleleft \tau$.
  \item For all nonempty $\rho, \sigma$ with $\rho \subsetneq \sigma
  \subseteq \tau$, we have $g_{\alpha_\rho} \leq g_{\alpha_\sigma}$.
\end{itemize}
Lastly, let
\[
  \mathcal{A}_{n+1}(\tau) = (-1)^{n+1}\mathcal{C}_n(\tau) * \alpha_\tau.
\]

Let us pause to connect these definitions with the proof of
Theorem \ref{hechler_theorem_two}.  There $B$ was an unbounded subset of $\kappa$. For $\rho \in [B]^2$, we set $\psi_\rho$ equal to $\mathcal{A}_2(\rho)$. Then by deducing equation
(\ref{certification1}) for an arbitrary $\tau \in [B]^3$, we showed that this assignment of values was as desired.
In the language just introduced, equation (\ref{certification1}) translates  to $\mathcal{C}_2(\tau) = 0$. Similarly, the expression $\mathbf{a}_{12} - \mathbf{a}_{02} + \mathbf{a}_{01} -
\mathbf{a}_{012}$ corresponds to $-\mathcal{S}_2(\tau)$. By the definition of $B$, the condition
$\mathfrak{u}_2(\tau)$ holds for all $\tau\in [B]^3$; the relation
$\mathcal{C}_2(\tau) = -\mathcal{S}_2(\tau)$ then followed immediately. The terms
in $\mathcal{S}_2(\tau)$ are all of the form $\mathsf{d} e(\vec{\beta})$, hence $\mathcal{S}_2(\tau) = 0$. This, in essence, was the deduction that $\mathcal{C}_2(\tau) = 0$.

The following two lemmas are easily proven by induction, so their proofs are left
to the reader.

\begin{lemma} \label{expression_lemma_i}
  For all $n$ with $2 \leq n < \omega$, all $\rho \in [B]^n$, and all
  $\tau \in [B]^{n+1}$, the expressions $\mathcal{A}_n(\rho)$,
  $\mathcal{S}_n(\tau)$, and $\mathcal{C}_n(\tau)$ are all of the form
  \[
    \sum_{i < \ell} a_i e(\vec{\alpha}_i),
  \]
  where, for all $i < \ell$, we have
  \begin{itemize}
    \item $a_i \in \bb{Z}$,
    \item $\vec{\alpha}_i \in [\kappa]^{n+1}$, and
    \item $\vec{\alpha}_i$ is of the form $\sigma_0
    {^\frown} \langle \alpha_{\sigma_j} \mid 1 \leq j \leq (n+1 - |\sigma_0|)
    \rangle$,
    where
    \begin{itemize}
      \item $0 < |\sigma_0| \leq n + 1$ and,
      if $|\sigma_0| < n + 1$, then $1 < |\sigma_1| \leq n+1$,
      \item if $|\sigma_0| < n + 1$, then $\sigma_0 \subseteq \sigma_1$,
      \item $\sigma_j \subsetneq \sigma_{j+1}$ for all $1 \leq j < (n+1 -
      |\sigma_0|)$, and
      \item letting $k = (n + 1 - |\sigma_0|)$, we have $\sigma_k \subseteq \rho$,
      in the case of $\mathcal{A}_n(\rho)$, and $\sigma_k \subseteq \tau$,
      in the case of $\mathcal{S}_n(\tau)$ or $\mathcal{C}_n(\tau)$.
    \end{itemize}
  \end{itemize}
\end{lemma}

The point of this lemma is that, although the expressions defining
$\mathcal{A}_n(\rho)$, $\mathcal{S}_n(\tau)$, $\mathcal{C}_n(\tau)$, and
$\mathfrak{u}_n(\tau)$ are constructed via a recursion referencing $\vec{\Phi}$, the values of
$\mathcal{A}_n(\rho)$, $\mathcal{S}_n(\tau)$, and $\mathcal{C}_n(\tau)$, and
the truth values of $\mathfrak{u}_n(\tau)$ only ever depend on $\Phi_n$
and not on $\Phi_m$ for any $m \neq n$. Hence we may meaningfully employ these expressions to argue as we do that $V^{\mathbb{P}}\vDash``\lim^n \mathbf{A} = 0\text{''}$ while only referencing a single family $\Phi_n$ at a time. Similarly, since the values of $\mathcal{A}_n(\rho)$,
$\mathcal{S}_n(\tau)$, and $\mathcal{C}_n(\tau)$ depend on the ordinals
$\langle\alpha_\sigma \mid \sigma \in [B]^{\leq n+1}\rangle$ but not on $\alpha_\sigma$ for
$|\sigma| > n+1$, verifications via these expressions that
$\lim^n \mathbf{A} = 0$ will never require that any ordinal $\alpha_\sigma$ of longer index ($|\sigma|>n+1)$ has yet been defined.

We can say more of $\mathcal{S}_n(\tau)$ than in
the previous lemma:

\begin{lemma} \label{expression_lemma_ii}
  For all $n$ with $2 \leq n < \omega$ and all $\tau \in [B]^{n+1}$,
  the expression $\mathcal{S}_n(\tau)$ is of the form
  \[
    \sum_{i < \ell} b_i \mathsf{d} e(\vec{\beta}_i),
  \]
  where $b_i \in \bb{Z}$ and $\vec{\beta}_i \in [\kappa]^{n+2}$ for all
  $i < \ell$. Since $\mathsf{d} e(\vec{\beta}) = 0$ for all
  $\vec{\beta} \in [\kappa]^{n+2}$, it follows that $\mathcal{S}_n(\tau) = 0$.
\end{lemma}

The following is the main technical lemma regarding these formal expressions.
Recall that we often interpret an element $\tau \in [\kappa]^{n+1}$ has an increasing
sequence, so, for example, the expression $I(\tau)$ denotes
$I\left(\bigwedge_{\alpha \in \tau} g_\alpha \right)$.

\begin{lemma} \label{s_c_lemma}
  Suppose that $2 \leq n < \omega$, $\tau \in [B]^{n+1}$, and
  $\mathfrak{u}_n(\tau)$ holds. Then
  \[
    \mathcal{S}_n(\tau)=(-1)^{n+1}\mathcal{C}_n(\tau).
  \]
\end{lemma}

\begin{proof}
  We proceed by induction on $n$, in fact establishing the following
  strengthening of the lemma, which we maintain as an inductive hypothesis:
  \begin{align*}
    \mathfrak{u}_n(\tau) & \text{ induces a reduction of } \mathcal{S}_n(\tau)
    \text{ to } (-1)^{n+1}\mathcal{C}_n(\tau) \text{ via } \\
    \text{(1)} & \text{ the cancellation of identical terms appearing
    with opposite sign} \\ & \text{ in the expansion of } \mathcal{S}_n(\tau), \\
    \text{(2)} & \text{ the cancellation of terms of the form } e(\vec{\alpha}
    [\vec{\sigma}]) \text{ in which } \vec{\sigma} \vartriangleleft \tau \text{ is }
    \\ & \text{ a long string, on the principle that, in the presence of } \mathfrak{u}_n(\tau),  \\ & \text{ these terms are
    all equal when restricted to the intersection} \\ & \text{ of their domains}.
  \end{align*}
  Such cancellations will be referred to as cancellations of type (1) and type (2),
  respectively.

  These cancellations could conceivably alter the domain of $\mathcal{S}_n(\tau)$. We first show that they do not; we show in particular that when $\mathfrak{u}_n(\tau)$ holds, the domain of both $\mathcal{S}_n(\tau)$  and $\mathcal{C}_n(\tau)$ is $I(\tau)$.
  By Lemma \ref{expression_lemma_i}, every term in each of $\mathcal{S}_n(\tau)$
  and $\mathcal{C}_n(\tau)$ is an integer multiple of $e(\vec{\alpha})$,
  where $\vec{\alpha}$ is of the form $\vec{\alpha}(\sigma_0) ^\frown
  \langle \alpha_{\sigma_1}, \cdots, \alpha_{\sigma_k} \rangle$, where
  the sequence $\langle \sigma_0, \cdots, \sigma_k \rangle$ is as described therein.
  In particular, $\sigma_0 \subseteq \sigma_j \subseteq \tau$ for all
  $j \leq k$. The condition $\mathfrak{u}_n(\tau)$ then ensures that $g_\eta \leq
  g_{\alpha_{\sigma_j}}$ for all $\eta \in \sigma_0$ and $j \leq k$.
  It follows that $\dom(e(\vec{\alpha})) = I(\sigma_0)
  \supseteq I(\tau)$ and hence that the domains of
  $\mathcal{S}_n(\tau)$ and $\mathcal{C}_n(\tau)$ both contain
  $I(\tau)$. Since both expressions explicitly include
  the term $e(\tau)$, whose domain is $I(\tau)$, the desired equalities hold.

  Now observe that the expression $\mathcal{S}_n(\tau)$ can be rewritten as
  follows:
  \begin{align*}
  \mathcal{S}_n(\tau)
     = & \;\mathsf{d} e(\tau,\alpha_{\tau}) - \sum_{i<n+1}(-1)^i \mathsf{d} [\mathcal{A}_{n}(\tau^i)*\alpha_\tau ] \\
     = & \;\mathsf{d} e(\tau,\alpha_{\tau})-\sum_{i<n+1}(-1)^i \sum_{j<n+2} (-1)^j [\mathcal{A}_{n}(\tau^i)*\alpha_\tau]^j \\
     = & \; \mathsf{d} e(\tau,\alpha_{\tau})-\sum_{i<n+1}(-1)^i \sum_{j<n+1} (-1)^j [\mathcal{A}_{n}(\tau^i)*\alpha_\tau]^j\\
     & \; -\sum_{i<n+1}(-1)^{n+i+1} \mathcal{A}_{n}(\tau^i) \\
     = & \; \sum_{i<n+1}(-1)^i e(\tau^i,\alpha_{\tau})-\sum_{i<n+1}(-1)^i \sum_{j<n+1} (-1)^j [\mathcal{A}_{n}(\tau^i)*\alpha_\tau]^j\\
     & \; +(-1)^{n+1}e(\tau)-\sum_{i<n+1}(-1)^{n+i+1} \mathcal{A}_{n}(\tau^i) \\
  \tag{$\star$}   = & \; \sum_{i<n+1}(-1)^i e(\tau^i,\alpha_{\tau})-\sum_{i<n+1}(-1)^i \sum_{j<n+1} (-1)^j [\mathcal{A}_{n}(\tau^i)*\alpha_\tau]^j\\
     & \; +(-1)^{n+1}\mathcal{C}_n(\tau)
  \end{align*}
  If type (1) and type (2) cancellations reduce the expression on the starred line to zero, then we will have shown
  $\mathcal{S}_n(\tau)=(-1)^{n+1}\mathcal{C}_n(\tau)$ while maintaining our inductive hypothesis. This is our goal. We begin with the base case of $n=2$. This amounts simply to a more careful translation of the end of the proof of
  Theorem \ref{hechler_theorem_two} into our current terminology.

Since $\mathcal{A}_2(\tau^i) * \alpha_\tau
  = e(\tau^i, \alpha_{\tau^i}, \alpha_\tau)$ for any  fixed $i < 3$ and hence
  $[\mathcal{A}_2(\tau^i) * \alpha_\tau]^2 = e(\tau^i, \alpha_\tau)$, the starred line in this case rearranges to
  \begin{align*}
    \sum_{i < 3} (-1)^i e(\tau^i, \alpha_\tau) - \sum_{i < 3}
    (-1)^{i} e(\tau^i, \alpha_\tau) -& \sum_{i < 3} (-1)^i \sum_{j < 2}
    (-1)^j[\mathcal{A}_2(\tau^i) * \alpha_\tau]^j \\
    = & \sum_{i < 3} (-1)^{i+1} \sum_{j < 2} (-1)^j[\mathcal{A}_2(\tau^i) * \alpha_\tau]^j
  \end{align*}
 The only cancellation above is of type (1).

  Fix $i < 3$ and $j < 2$ and observe that $[\mathcal{A}_2(\tau^i)
  * \alpha_\tau]^j$ is of the form $e(\eta, \alpha_{\tau^i}, \alpha_\tau)$,
  with $\eta \in \tau^i$. But then $\vec{\sigma} = \langle \{\eta\},
  \tau^i, \tau \rangle$ is a long string for $\tau$ and $e(\eta, \alpha_{\tau^i},
  \alpha_\tau) = e(\vec{\alpha}[\vec{\sigma}])$. Hence by
  $\mathfrak{u}_2(\tau)$ there is a fixed $\epsilon$ such that the restriction of each such term to the common domain $I(\tau)$ is equal to $\epsilon$.
  The starred line thus reduces further to
  \[
    \sum_{i < 3} (-1)^{i+1} \sum_{j < 2} (-1)^j[\mathcal{A}_2(\tau^i) * \alpha_\tau]^j
    = \sum_{i < 3} (-1)^{i + 1} (\epsilon - \epsilon) = 0,
  \]
  as desired. The reduction is evidently by way of cancellations of type (2).
  This completes the base case.

  Now assume that $n > 2$ and that the inductive hypothesis holds for
  $n - 1$. We need two claims.

  \begin{claim}\label{AisS}
    For all $i < n+1$,
    \[
      \sum_{j<n+1} (-1)^j [\mathcal{A}_{n}(\tau^i)*\alpha_\tau]^j=(-1)^n
      \mathcal{S}_{n-1}(\tau^i)*\alpha_\tau.
    \]
  \end{claim}

  \begin{proof}
    Rewrite the sum $\sum_{j<n+1} (-1)^j [\mathcal{A}_{n}(\tau^i)*\alpha_\tau]^j$
    in the following sequence of steps; each simply consists of an application
    of the definitions of the expressions under consideration:
    \begin{equation*}
      \begin{split}
        & \,\sum_{j<n+1} (-1)^j [(-1)^n\mathcal{C}_{n-1}
        (\tau^i)*\alpha_{\tau^i}]^j*\alpha_\tau\\
        = & \,(-1)^n\sum_{j<n+1} (-1)^j\left[ [e(\tau^i) -
        \sum_{k<n}(-1)^{k} \mathcal{A}_{n-1}
        ((\tau^i)^k)]*\alpha_{\tau^i}\right]^j*\alpha_\tau \\
        = & \,(-1)^n\left[\mathsf{d} e(\tau^i,\alpha_{\tau^i})-
        \sum_{k<n}(-1)^k \mathsf{d} [\mathcal{A}_{n-1}((\tau^i)^k)*\alpha_{\tau^i}]
        \right]*\alpha_\tau \\
        = & \,(-1)^n\mathcal{S}_{n-1}(\tau^i)*\alpha_\tau
      \end{split}
    \end{equation*}
  \end{proof}

  \begin{claim}\label{SisC}
    For all $i < n + 1$, the condition $\mathfrak{u}_n(\tau)$ implies that
    \[
      (-1)^n\mathcal{S}_{n-1}(\tau^i)*\alpha_\tau =
      \mathcal{C}_{n-1}(\tau^i)*\alpha_\tau
    \]
    via cancellations of type (1) and of type (2).
  \end{claim}

  \begin{proof}
    By the inductive hypothesis applied to $\mathcal{S}_{n-1}(\tau^i)$, we know
    that the terms appearing in $\mathcal{S}_{n-1}(\tau^i)$ but not in
    $(-1)^n\mathcal{C}_{n-1}(\tau^i)$ pair off in pairs either of type (1), meaning that the pair consists of identical terms
    with opposite signs, or of type (2), meaning that the pair is of the form
    $\big(e(\vec{\alpha}[\vec{\sigma}_0]), -e(\vec{\alpha}[\vec{\sigma}_1])\big)$,
    where $\vec{\sigma}_0$ and $\vec{\sigma}_1$ are long strings for $\tau^i$.

    If such a pair $(t, -t)$ is of type (1), then $t * \alpha_\tau$ and
    $-t * \alpha_\tau$ also form a type (1) canceling pair in $\mathcal{S}_{n-1}(\tau^i) * \alpha_\tau$. Notice also that if $\vec{\sigma}$ is a long
    string for $\tau^i$ then $\vec{\sigma}^* = \vec{\sigma}^\frown \langle
    \tau \rangle$ is a long string for $\tau$ and $\vec{\alpha}[\vec{\sigma}^*]
    = \vec{\alpha}[\vec{\sigma}] ^\frown \langle \alpha_\tau \rangle$.
    Therefore, if $\big(e(\vec{\alpha}[\vec{\sigma}_0]),
    -e(\vec{\alpha}[\vec{\sigma}_1])\big)$ is a type (2) pair of terms from
    $\mathcal{S}_{n-1}(\tau^i)$ then
    $e(\vec{\alpha}[\vec{\sigma}_0]) * \alpha_\tau = e(\vec{\alpha}[\vec{\sigma}_0^*])$
    and $-e(\vec{\alpha}[\vec{\sigma}_1]) * \alpha_\tau = -e(\vec{\alpha}
    [\vec{\sigma}_1^*])$. The condition $\mathfrak{u}_n(\tau)$ then implies that these terms are
    equal to some $\epsilon$ and $-\epsilon$, respectively, and thus cancel in
    $\mathcal{S}_{n-1}(\tau^i) * \alpha_\tau$ in a type (2) manner.

    The terms of $\mathcal{S}_{n-1}
    (\tau^i) * \alpha_\tau$ that remain after these cancellations are precisely those of the form $t * \alpha_\tau$, where $t$
    appears in $(-1)^n\mathcal{C}_{n-1}(\tau^i)$. Hence
    $\mathcal{S}_{n-1}(\tau^i) * \alpha_\tau = (-1)^n \mathcal{C}_{n-1}(\tau^i)
    * \alpha_\tau$. Hence $(-1)^n \mathcal{S}_{n-1}(\tau^i) * \alpha_\tau
    = \mathcal{C}_{n-1}(\tau^i)$ via cancellations of type (1) and of type (2).
  \end{proof}

  By Claims \ref{AisS} and \ref{SisC} the starred line above now reduces to
  the following:
  \begin{equation*}
    \begin{split}
      &\,\sum_{i<n+1}(-1)^i e(\tau^i,\alpha_{\tau})-\sum_{i<n+1}(-1)^i
      [\mathcal{C}_{n-1}(\tau^i)*\alpha_\tau ] \\
      = &\,\sum_{i<n+1}(-1)^i e(\tau^i,\alpha_{\tau})-\sum_{i<n+1}(-1)^i
      \left[ e(\tau^i,\alpha_{\tau})-\sum_{j<n}(-1)^j\mathcal{A}_
      {n-1}((\tau^i)^j)*\alpha_\tau\right]\\
      = &\,\sum_{i<n+1}\sum_{j<n}(-1)^{i+j}\mathcal{A}_{n-1}((\tau^i)^j)*\alpha_\tau
    \end{split}
  \end{equation*}
  Observe that the only cancellation above is of type (1).
  Now the key observation is that to each $(\tau^i)^j$ with $i\leq j<n$ there
  corresponds a $(\tau^{i'})^{j'}$ with $j'<i'<n+1$, and vice versa. Namely, when
  $i'=j+1$ and $j'=i$ then $(\tau^i)^j=(\tau^{i'})^{j'}$. In the above sum, these each
  associate to the signs $(-1)^{i+j}$ and $(-1)^{i+j+1}$, respectively. We therefore
  conclude by  rewriting that sum as
  \[
    \sum_{i\leq j<n}(-1)^{i+j}\mathcal{A}_{n-1}((\tau^i)^j)*\alpha_\tau+
    \sum_{j<i<n+1}(-1)^{i+j}\mathcal{A}_{n-1}((\tau^i)^j)*\alpha_\tau=0,
  \]
  where again all cancellations are of type (1). This completes the induction step. It therefore completes the proof of Lemma \ref{s_c_lemma}.
\end{proof}

We are now ready to prove our main result, which we restate here for convenience.

\begin{THM}
Let $\kappa\in V$ be a weakly compact cardinal, and let $\mathbb{P}$ denote a length-$\kappa$ finite-support iteration of Hechler forcings. Then $$V^{\mathbb{P}}\vDash\textnormal{``}\,\mathrm{lim}^{n}\mathbf{A}=0\text{ for all }n>0.\textnormal{''}$$
\end{THM}

\begin{proof}
  We will show that any $n$-coherent family of functions in $V^{\mathbb{P}}$ is trivial.
  Since we have already done so for $n = 1$, we will assume in what follows that $n > 1$.
  For all $\alpha < \kappa$, let
  $\dot{g}_\alpha$ be a nice $\bb{P}$-name for the Hechler real added at the $\alpha^{\textnormal{th}}$ stage of $\mathbb{P}$. As in the $n=1$ and $n=2$ cases, it will suffice to consider $n$-coherent families in $V^{\bb{P}}$ indexed
  by $n$-tuples of functions from $\langle \dot{g}_\alpha \mid \alpha < \kappa \rangle$.
  Therefore fix a condition $p_0 \in \bb{P}$ and a family of $\bb{P}$-names
  $\dot{\Phi}_n = \left\langle \dot{\varphi}_{\vec{\alpha}}:\dot{I}(\vec{\alpha}) \rightarrow \bb{Z} ~ \middle| ~ \vec{\alpha}\in [\kappa]^n\right\rangle$ forced
  by $p_0$ to be an $n$-coherent family of functions. We will find a $q \leq p_0$ which forces
  $\dot{\Phi}_n$ to be trivial.

  Let $A_0 = \kappa \setminus (\max(\dom(p_0)) + 1)$. Much as before, for all $\vec{\alpha}\in[A_0]^{n+1}$
  there exists a condition $q_{\vec{\alpha}} \leq p_0$ such that
  $q_{\vec{\alpha}} \Vdash_{\bb{P}} ``\dot{g}_{\alpha_0} \leq \dots
  \leq \dot{g}_{\alpha_i} \leq \dots \leq \dot{g}_{\alpha_{n}}"$ and
  $q_{\vec{\alpha}}$ decides the value of
  $\dot{\mathtt{e}}(\vec{\alpha})$ to be  equal to some $\epsilon_{\vec{\alpha}} \in V$.
  Thin out $A_0$ to an unbounded $A_1$ such that
  $\epsilon_{\vec{\alpha}}$ equals some fixed $\epsilon$ for all
  $\vec{\alpha} \in [A_1]^{n+1}$.

  Now apply Lemma \ref{uniformizing_lemma} to $\left\langle q_{\vec{\alpha}} ~ \middle| ~ \vec{\alpha} \in [A_1]^{n+1} \right\rangle$ to find an unbounded $A \subseteq A_1$
  together with sets $\left\langle u_{\vec{\alpha}} ~ \middle| ~ \vec{\alpha} \in [A]^{\leq n}
  \right\rangle$, a natural number $\ell$, and stems $\langle s_i \mid i < \ell \rangle$
  as in the statement of the lemma.
  Next, define conditions $\left\langle q_{\vec{\alpha}} ~ \middle| ~ \vec{\alpha}
  \in [A]^{\leq n} \right\rangle$
  as follows: for each $\vec{\alpha}$ in $[A]^{\leq n}$ let $\vec{\beta}$ be
  an element of $[A]^{n+1}$ such that $\vec{\alpha} \sqsubseteq \vec{\beta}$
  and let $q_{\vec{\alpha}} = q_{\vec{\beta}} \restriction u_{\vec{\alpha}}$.
  As usual, by item (3) of Lemma \ref{uniformizing_lemma}, these definitions
  are independent of all of our choices of $(n+1)$-tuples $\vec{\beta}$.
  We will sometimes abuse notation and write, for example, $q_{\alpha}$ instead of
  $q_{\langle \alpha \rangle}$. 

  Let $q = q_\emptyset$, and note that $q \leq p_0$. We claim that $q$ forces $\dot{\Phi}$ to be trivial.
  This we argue by first partitioning $A$ into $n + 1$ disjoint and unbounded subsets $\langle \Gamma_i\mid 1\leq i\leq n+1\rangle$. Let $\dot{B}$ be a $\bb{P}$-name for the set of
  $\alpha \in \Gamma_1$ such that $q_{\alpha} \in \dot{G}$, where $\dot{G}$ is the canonical name for the $\mathbb{P}$-generic filter. By exactly the same reasoning as in the proofs of Theorem \ref{hechler_theorem_one} and Theorem \ref{hechler_theorem_two},
  \[
    q \Vdash_{\bb{P}}``\dot{B} \text{ is unbounded in } \kappa."
  \]

 \begin{claim}\label{strings}
   Fix $\tau\in [\kappa]^{n+1}$. The condition $q$ forces the following to hold
   in $V^{\mathbb{P}}:$

   Suppose that $m$ and $\left\langle\alpha_\sigma ~ \middle| ~ \sigma\in[\tau]^{<m}\textnormal{
   and }\sigma\neq\emptyset\right\rangle$ are such that
   \begin{itemize}
     \item $1< m\leq n+1$,
     \item $\alpha_{\{\eta\}} = \eta$ for all $\eta \in \tau$,
     \item $\alpha_\rho<\alpha_\sigma$ whenever $\rho$ is a proper subset of $\sigma$,
     \item $\alpha_\sigma\in \Gamma_{|\sigma|}$ for all nonempty
     $\sigma\in[\tau]^{<m}$, and
     \item for any $1\leq\ell< m$ and subset-initial segment $\vec{\sigma}
     \vartriangleleft \tau$ of length $\ell$, we have $q_{\vec{\alpha}[\vec{\sigma}]}
     \in \dot{G}$. In particular, $\eta \in \dot{B}$ for all $\eta \in \tau$.
   \end{itemize}
   Then there exists a sequence $\left\langle\alpha_\sigma ~ \middle| ~ \sigma\in[\tau]^m\right\rangle$ of elements of $\Gamma_m$ which
   together with $\left\langle\alpha_\sigma ~ \middle| ~ \sigma\in([\tau]^{<m}\backslash\{\emptyset\})\right\rangle$
   satisfies
   \begin{itemize}
     \item $\alpha_\rho<\alpha_\sigma$ whenever $\rho$ is a proper subset of $\sigma$, and
     \item for any subset-initial segment $\vec{\sigma} \vartriangleleft \tau$
     of length $m$, we have $q_{\vec{\alpha}[\vec{\sigma}]} \in \dot{G}$.
   \end{itemize}
 \end{claim}

 \begin{proof}
   Fix an $m$ and $\left\langle\alpha_\sigma ~ \middle| ~ \sigma\in[\tau]^{<m} \backslash \{\emptyset\}\right\rangle$
  and an $r\leq q$ such that $r$ forces $m$ and $\left\langle\alpha_\sigma ~ \middle| ~
  \sigma\in[\tau]^{<m} \backslash \{\emptyset\}\right\rangle$ to be as in the premise of
  the claim. In particular, $r$ forces that $q_{\vec{\alpha}[\vec{\sigma}]}$ is in
  $\dot{G}$ for every $\vec{\sigma}\vartriangleleft\tau$ as in the premise of the claim.
  As before, we may assume that $r$ extends all such $q_{\vec{\alpha}[\vec{\sigma}]}$.
  For each $\sigma\in[\tau]^m$ and $\vec{\sigma}\vartriangleleft\sigma$ of length $m-1$,
  observe that the sequence $\left\langle u_{\vec{\alpha}[\vec{\sigma}]^\frown\gamma}
  \backslash u_{\vec{\alpha}[\vec{\sigma}]}~ \middle| ~ \gamma\in\Gamma_m\backslash
  (\max(\vec{\alpha}[\vec{\sigma}])+1)\right\rangle$ consists of pairwise disjoint sets.
  Therefore, we can choose an $\alpha_\sigma\in\Gamma_m$
  so that $\left(u_{\vec{\alpha}[\vec{\sigma}]^\frown\alpha_\sigma} \backslash
  u_{\vec{\alpha} [\vec{\sigma}]}\right)\cap\text{dom}(r)=\emptyset$ for all
  $\vec{\sigma}\vartriangleleft \sigma$ of length $m-1$. This then defines a
  family of pairwise-aligned sequences $\langle\vec{\alpha}[\vec{\sigma}]\mid\vec{\sigma}
  \vartriangleleft\tau \textnormal{ is of length }m\rangle$, since for any
  $\alpha\in\Gamma_k$ in the intersection of two such $\vec{\alpha}[\vec{\sigma}]$
  and $\vec{\alpha}[\vec{\rho}]$,
  \[
    |\vec{\alpha}[\vec{\sigma}]\cap\alpha|=|\vec{\alpha}[\vec{\rho}]\cap\alpha|=k.
  \]
  As before, this implies that the family $\left\langle q_{\vec{\alpha}[\vec{\sigma}]} ~ \middle| ~
  \vec{\sigma}\vartriangleleft\tau\text{ is of length }m\right\rangle$ consists of
  pairwise compatible conditions, with identical stems wherever their domains intersect.
  Applying Lemma \ref{compatibility_lemma} to this collection together with
  $r$ then yields a lower bound $s\in\mathbb{P}$ for all of these conditions,
  thereby witnessing the conclusion of the claim.
 \end{proof}

 As in the proof of Theorem \ref{hechler_theorem_two}, we will conclude our argument
 in $V[G]$, where $G$ is a $\bb{P}$-generic filter containing $q$. Let $B$
 be the realization of $\dot{B}$ and $\Phi_n$ the realization of
 $\dot{\Phi}_n$ in $V[G]$. For each nonempty $\sigma \in
 [B]^{\leq n+1}$, we will specify an ordinal $\alpha_\sigma < \kappa$ in such a way
 that the collection $\left\langle\alpha_\sigma ~ \middle| ~ \sigma \in [B]^{\leq n+1}\right\rangle$ satisfies
 the following:
 \begin{itemize}
   \item $\alpha_{\{\eta\}} = \eta$ for all $\eta \in B$,
   \item $\alpha_\rho < \alpha_\sigma$ whenever $\rho$ is a proper subset of
   $\sigma$,
   \item $\alpha_\sigma \in \Gamma_{|\sigma|}$ for all nonempty $\sigma \in
   [B]^{\leq n + 1}$, and
   \item for every $\tau \in [B]^{n+1}$ and every subset-initial segment
   $\vec{\sigma} \vartriangleleft \tau$, we have $q_{\vec{\alpha}[\vec{\sigma}]}
   \in G$.
 \end{itemize}

 The construction of $\left\langle\alpha_\sigma ~ \middle| ~ \sigma \in [B]^{\leq n+1}\right\rangle$
 proceeds via a straightforward recursion on $|\sigma|$, invoking
 Claim \ref{strings}.

 The $n$-coherent family $\Phi_n$ and collection $\left\langle\alpha_\sigma
 ~ \middle| ~ \sigma \in [B]^{\leq n + 1} \backslash \{\emptyset\}\right\rangle$ determine
 expressions $\mathcal{A}_n(\rho)$ for all $\rho \in [B]^n$ and
 expressions $\mathcal{S}_n(\tau)$ and $\mathcal{C}_n(\tau)$ and statements
 $\mathfrak{u}_n(\tau)$ for all $\tau \in [B]^{n+1}$. Note that our choices of
  $q$, $B$, and the sequence $\left\langle\alpha_\sigma ~ \middle| ~ \sigma \in [B]^{\leq n + 1}\right\rangle$
 ensure that $\mathfrak{u}_n(\tau)$ holds for all $\tau \in [B]^{n+1}$.

 For each $\rho \in [B]^n$, let $\psi_\rho = \mathcal{A}_n(\rho)$.
 It follows from the definition of $\mathcal{A}_n(\rho)$ and the fact that $\mathfrak{u}_n(\tau)$
 holds for all $\tau \in [B]^{n+1}$ that $\mathcal{A}_n(\rho)$
 is a finitely-supported function with domain
 $I(\rho)$. To
 show that $\Phi_n$ is trivial, it will suffice to show that
 \begin{align}\label{last}
  e(\tau) = \sum_{i < n+1} (-1)^i \psi_{\tau^i}
 \end{align}
 for all
 $\tau \in [B]^{n+1}$. Fix such a $\tau$. The above equation is easily seen to be equivalent to the assertion that $\mathcal{C}_n(\tau) = 0$.
 By Lemma \ref{s_c_lemma} and the fact that $\mathfrak{u}_n(\tau)$ holds, $\mathcal{S}_n(\tau) = (-1)^{n+1}\mathcal{C}_n(\tau)$.
 By Lemma \ref{expression_lemma_ii}, $\mathcal{S}_n(\tau) = 0$. In consequence, $\mathcal{C}_n(\tau) = 0$. This shows equation (\ref{last}) for an arbitrary $\tau$ hence $\Phi_n$ is trivial.
\end{proof}

\section{Conclusion}\label{conclusion}

Several natural questions follow immediately from the above results. The first of these is whether the assumption of a weakly compact cardinal is necessary for
the conclusion of the Main Theorem. As noted in Remark~\ref{recent_work_remark} above,
this was answered in the negative by the authors together with Michael Hru\v{s}\'{a}k in \cite{svhdlwolc},
where they showed that $\lim^n \mathbf{A} = 0$ for all $n > 0$ after the addition of $\beth_\omega$-many Cohen reals.
Closely related is the question of whether a large continuum is necessary for the conclusion of the Main Theorem.
\begin{question} \label{continuum_q}
What is the minimum value of the continuum compatible with the statement $``\lim^n\mathbf{A}=0 \text{ for all }n>0"$?
\end{question}
By the results of \cite{mp}, that minimum value is at least $\aleph_2$. By the
aforementioned result in \cite{svhdlwolc}, $\aleph_{\omega+1}$ is an upper bound
for the answer to Question~\ref{continuum_q}.

Our Main Theorem also lends the original, motivating questions a certain renewed charge:
\begin{question}\label{motivating} Is it consistent with the $\mathsf{ZFC}$ axioms (possibly modulo large cardinal assumptions) that strong homology is additive
\begin{itemize}
\item on Polish spaces?
\item on locally compact metric spaces?
\item on metric spaces?
\end{itemize}
\end{question}
The extension $V^{\mathbb{P}}$ of the Main Theorem is, of course, a candidate model for an affirmative answer to any of these questions, and indeed, as noted in Remark~\ref{recent_work_remark},
Bannister, Bergfalk, and Moore recently showed in \cite{bannister_bergfalk_moore} that strong homology is additive
on the class of locally compact Polish spaces in $V^{\bb{P}}$. A word is in order here about the machinery of strong homology, which consists first in an assignment of a system of approximations to a given topological space $X$ and second in an assignment of a homology group to $X$ by way of the system of homology groups of its approximations \cite{SSH}. Compact metric spaces figure within this framework as particularly tractable: they admit sequential systems of approximations. Indeed, the index-set $^\omega\omega$ so central to all our considerations above arose as a countable product (induced by a countable topological sum) of just such a family of height-$\omega$ systems of approximations to $n$-dimensional Hawaiian earrings; the argument of \cite{bannister_bergfalk_moore} consists largely in correlating a broader class of ${^\omega}\omega$-indexed inverse systems to strong homology computations for locally compact Polish spaces, and in then showing that the present work's arguments apply to that broader class. In short, strong homology computations translate the ways that each of the bulleted classes above relate to ``simpler'' spaces to an associated family of index-sets; Question \ref{motivating} is perhaps largely one of what sorts of combinatorics these index-sets may or may not simultaneously support. Accordingly, the combinatorics of the interrelations among various partial orders may well play some role in its further solution (see, e.g., \cite{dircof}). Through these interrelations, the vanishing of one system's higher derived limits may entail the vanishing of others'; Theorem 5.1 of \cite{b} is an example of such a result.

At some broader level, Question \ref{motivating} is asking what sorts of continuities we may compatibly expect of homology functors on categories of topological spaces properly extending $\mathsf{HCW}$, i.e., properly including the category of spaces homotopy equivalent to a CW-complex. For the question \emph{on what class of spaces may a homology theory be both strong shape invariant and additive?} \hspace{-.07 cm}is at heart a question about the interactions of the inverse and direct limits associated to the first prospective property and the second, respectively. As mentioned, closely related is the question of whether the strong homology groups of a space are the direct limits of the strong homology groups of its compact subspaces; a highly canonical extension of Steenrod homology exists on any class for which the answer is yes. This is the case in \cite{bannister_bergfalk_moore}; see its introduction and \cite{mp} and \cite[Chapter 21.5]{SSH} for further discussion.

Lastly, one of this paper's referees asked what happens if we replace our iteration $\bb{P}$ of
Hechler forcing with the standard length-$\kappa$ finite support iteration
$\bb{Q}$ for forcing Martin's Axiom ($\mathsf{MA}$) together with $2^{\aleph_0} = \kappa$.
We conjecture that, for all $n > 0$, $\lim^n \mathbf{A} = 0$ will hold in $V^{\bb{Q}}$ as well, but our argument does not
directly generalize to prove this, since it makes essential use of the fact
that Hechler forcing is $\sigma$-centered (in particular, any finitely many
conditions with the same stem have a common lower bound). This is evident in
Lemmas \ref{compatibility_lemma} and \ref{uniformizing_lemma}(1), which are
then used in the proof of Claim~\ref{strings}. There are straightforward
generalizations of Lemmas \ref{compatibility_lemma} and \ref{uniformizing_lemma}
to arbitrary length-$\kappa$ finite support iterations of $\sigma$-centered
forcings of size ${<}\kappa$. (In this more general context, if
$\bb{S}$ is a $\sigma$-centered poset and $p \in \bb{S}$, then $s_p$,
the \emph{stem} of $p$, is interpreted as the unique $j < \omega$ such that
$p \in S_j$, where $\langle S_j \mid j < \omega \rangle$ is a fixed partition
of the underlying set of $\bb{S}$ into centered subsets.) With this in mind,
our proof adapts to show that, if $\kappa$ is weakly compact and $\bb{Q}^*$ is the standard
length-$\kappa$ finite support iteration for forcing $\mathsf{MA(\sigma\text{-centered})} +
2^{\aleph_0} = \kappa$, then $\lim^n \mathbf{A} = 0$ for all $n > 0$.
The proof proceeds by considering $n$-coherent families indexed by $n$-tuples of
the generic reals $\langle \dot{g}_\alpha \mid \alpha \in B \rangle$,
where $B \in V$ is an unbounded subset of $\kappa$ such that, for each
$\alpha \in B$, the $\alpha^{\mathrm{th}}$ iterand of $\bb{Q}^*$ is a
$\bb{Q}^*_\alpha$-name for Hechler forcing and $\dot{g}_\alpha$ is a name
for the corresponding Hechler real. We therefore close with the question of the full form of Martin's Axiom, and with thanks to the referee for asking it:

\begin{question}
  Suppose that $\kappa$ is a weakly compact cardinal and that $\bb{Q}$ is the standard
  length-$\kappa$ finite support iteration for forcing $\mathsf{MA} + 2^{\aleph_0} = \kappa$.
  Is it the case that, in $V^{\bb{Q}}$, $\lim^n \mathbf{A} = 0$ for all $n > 0$?
\end{question}

\textbf{Acknowledgements.} This work began with a visit by Chris Lambie-Hanson to UNAM Morelia. The authors would like to thank both the institutions VCU and UNAM for their support for this visit. More particularly, many of this paper's most fundamental impulses are due to Michael Hru\v{s}\'{a}k and grew out of conversations with him both before and during Chris's visit; the authors would like to thank him very especially. Jeffrey Bergfalk would like to thank Justin Moore and Jim West as well for much patient and formative instruction in the subject-matter of this paper. The authors would also like to thank Stevo Todorcevic for encouraging them to reduce their large-cardinal assumption from a measurable to a weakly compact, and for his several suggestions for how to do so. Finally, they would like to thank both referees for their extremely thorough and thoughtful readings, comments, and questions.

\bibliographystyle{amsplain}
\bibliography{bibarx3}

\end{document}